\documentclass[a4paper,10pt, twoside]{article}
\usepackage[a4paper,left=3cm,right=3cm,top=2.8cm,bottom=2.8cm,bindingoffset=5mm]{geometry}
\usepackage{mathpazo}

\usepackage[utf8]{inputenc} 
\usepackage[T1]{fontenc}  	
\usepackage[misc]{ifsym}  	
\usepackage{verbatim}   	

\usepackage{amsmath}
\usepackage{amsthm}
\usepackage{amssymb}
\usepackage{amsfonts}
\usepackage{mathrsfs}

\usepackage{framed}
\usepackage{graphicx}
\usepackage{subfigure}
\usepackage{color}
\usepackage{pgfplots}

\usepackage{cite}
\usepackage{algorithmic}
\usepackage{algorithm}
\usepackage{listings} 
\usepackage{tabularx}
\usepackage{multirow}
\usepackage{enumitem}

\usepackage[colorlinks,citecolor=blue, urlcolor = green!50!black]{hyperref}

\newtheorem{theorem}{Theorem}[section]
\newtheorem{corollary}[theorem]{Corollary}
\newtheorem{lemma}[theorem]{Lemma}
\newtheorem{assumption}[theorem]{Assumption}
\newtheorem{definition}{Definition}[section]
\newtheorem{remark}{Remark}

\renewcommand{\abstractname}{Abstract}

\numberwithin{equation}{section}

\newcommand{\Uad}{{U_{\mathrm{ad}}}}
\newcommand{\Uadnu}{{U^\nu_{\mathrm{ad}}}}

\newcommand{\Anua}{\mathcal{A}^{\nu,a}}
\newcommand{\Anub}{\mathcal{A}^{\nu,b}}
\newcommand{\Inu}{\mathcal{I}^\nu}
\newcommand{\Aset}{\mathcal{A}}
\newcommand{\Iset}{\mathcal{I}}
\newcommand{\Yset}{\mathcal{Y}}

\newcommand{\N}{\mathbb{N}}
\newcommand{\R}{\mathbb{R}}

\newcommand{\CO}{{C(\bar{\Omega})}}

\newcommand{\U}{{L^2(\Omega)}}

\newcommand{\by}{\bar{y}}
\newcommand{\bu}{\bar{u}}
\newcommand{\bp}{\bar{p}}
\newcommand{\bmu}{\bar{\mu}}

\newcommand\norm[1]{\left\lVert#1\right\rVert}
\newcommand\dx{\,\mathrm{d}x}

\usepackage{fancyhdr}
\title{\large \textbf{\MakeUppercase{On Non-Reducible Multi-Player Control Problems and their Numerical Computation}} \thanks{The first author was supported by the German Research Foundation (DFG) within the priority program "Non-smooth and Complementarity-based Distributed Parameter Systems: Simulation and Hierarchical Optimization" (SPP 1962) under grant number Wa 3626/3-1 and the second author was supported under Wa 3626/1-1}}

\author{
	Veronika Karl%
	\thanks{\noindent Universität Würzburg, Institut für Mathematik,
		Emil-Fischer-Str.\ 30, 97074 Würzburg, Germany;\newline 
		\hspace*{1.9em} \Letter\href{mailto:veronika.karl@mathematik.uni-wuerzburg.de}{\texttt{ veronika.karl@mathematik.uni-wuerzburg.de}}}
       \and
       Frank P{\"o}rner%
       	\thanks{Universität Würzburg, Institut für Mathematik,
		Emil-Fischer-Str.\ 30, 97074 Würzburg, Germany;\newline
		\hspace*{1.9em} \Letter\href{mailto:frank.poerner@mathematik.uni-wuerzburg.de}{\texttt{ frank.poerner@mathematik.uni-wuerzburg.de}}  } 
}

\usepackage{titlesec}
\titleformat{\section}[runin]{\bfseries}{}{1em}{\thesection. }[. ]
\titleformat{\subsection}[runin]{\bfseries}{}{1em}{\thesubsection. }[. ]

\begin{document}
\maketitle

\vspace{-5ex}
\renewcommand{\abstractname}{\vspace{-\baselineskip}}
\begin{abstract}
\noindent \textbf{Abstract.} In this article we consider a special class of Nash equilibrium problems that cannot be reduced to a single player control problem. Problems of this type can be solved by a semi-smooth Newton method. Applying results from the established convergence analysis we derive superlinear convergence for the associated Newton method and the equivalent active-set method. We also provide detailed finite element discretizations for both methods. Several numerical examples are presented to support the theoretical findings.
 
\noindent 
	
	
	
\end{abstract}

\section{Introduction}
We consider a Nash Equilibrium Problem (NEP) in the optimal control setting. Here, $N\in\N$ denotes the number of players. The strategy space of all players is given by $U:=L^2(\Omega)^N$. The player $\nu\in\lbrace1,...,N\rbrace$ is in control of the variable $u^\nu\in L^2(\Omega)$. The strategies of all players, except the $\nu$-th player are denoted by $u^{-\nu}\in L^2(\Omega)^{N-1}$. Hence, we have the notation $u:=(u^\nu,u^{-\nu})$. 
Investigating multi-player control problems in the function space setting one usually assumes \cite{HintermSurowiec2013gnep,HintermSurowiecKaemm2015gnep,KanzowKarlSteckWachsmuth2017gnep} that the players' observation areas coincide. An exemplary problem setting for the $\nu$-th player's problem is given by
\begin{align*}
\underset{u^\nu\in L^2(\Omega)}\min \ &\frac{1}{2}\norm{Su-y_d^\nu}_{L^2(\Omega)}^2 +\frac{\alpha}{2}\norm{u^\nu}_{L^2(\Omega)}^2 \\
\text{s.t.}\quad &u^\nu \in \Uadnu,
\end{align*}
where $\Uadnu\subset L^2(\Omega)$ is a bounded convex set. In this setting existence and uniqueness of solutions are quite forward to show by exploiting standard arguments. Indeed the problem can be transformed into a convex single player control problem \cite[Proposition 3.10]{HintermSurowiecKaemm2015gnep}. However, the situation becomes considerably more complicated if the observation area of the tracking term differs for each player. To be more precise we consider $\Omega_\nu\subset\Omega$ and assume that the $\nu$-th player aims at solving
\begin{equation}\label{eq:NEP-offest-introduction}
\begin{split}
\underset{u^\nu\in L^2(\Omega)}\min\ &\frac{1}{2}\norm{Su-y_d^\nu}_{L^2(\Omega_\nu)}^2 +\frac{\alpha}{2}\norm{u^\nu}_{L^2(\Omega)}^2 \\
\text{s.t.}\quad &u^\nu \in \Uadnu. 
\end{split}
\end{equation}
We will give the precise setting below in Section \ref{sec:setting}. Problems of this type lack of the possibility to be reduced to a single control problem and therefore require a different treatment. To the best of our knowledge, until now, there exists no theory regarding the uniqueness of solutions of this kind of problems. By imposing an assumption on the regularization parameter $\alpha >0$ we will show existence and uniqueness of solutions of the NEP \eqref{eq:NEP-offest-introduction} in Theorem \ref{thm:NEPunique}. 
Solving NEPs is not only interesting for solving the problem itself. Moreover, solving generalized Nash equilibrium problems (GNEPs) that include inequality constraints like $Su\leq \psi$, $\psi\in \CO$ require in certain solution methods the solution of a sequence of NEPs \cite{HintermSurowiec2013gnep,KanzowKarlSteckWachsmuth2017gnep}. It is a natural approach to apply the semi-smooth Newton method in order to solve these multi-player control problems that are given by the following extension of \eqref{eq:NEP-offest-introduction} 
\begin{equation}\label{eq:intro2}
\begin{split}
\underset{u^\nu\in L^2(\Omega)}\min\ &\frac{1}{2}\norm{Su-y_d^\nu}_{L^2(\Omega_\nu)}^2 +\frac{\alpha}{2}\norm{u^\nu}_{L^2(\Omega)}^2 +\frac{1}{2\rho}\norm{(\mu+\rho(Su-\psi))_+}^2_{L^2(\Omega)} 
\\
\text{s.t.}\quad &u^\nu \in \Uadnu,
\end{split}
\end{equation}
where $\mu\in L^2(\Omega)$ and $\rho>0$ denotes a penalization parameter. Furthermore $(\cdot)_+ := \max(\cdot,0)$ in a pointwise almost everywhere sense. Here, we will focus on the studies of the corresponding semi-smooth Newton method. Since the tracking term is again considered on $\Omega_\nu$ only, the method can be expected to converge superlinear only if $\alpha$ is sufficiently large, see Theorem \ref{theo:convSSN-PDE}. \medskip\newline
The outline of this paper is as follows: In Section \ref{sec:setting} we introduce the reader to non-reducible NEPs and the extended augmented NEP. Here, our main results state existence and uniqueness of solutions, see Theorem \ref{thm:NEPunique} and Theorem \ref{thm:augNEP_unique}. In Section \ref{sec:SSNMgen} we collect results from the literature that are necessary for discussing superlinear convergence of the semi-smooth Newton method. Here, we contribute Lemma \ref{lemma:max_semismooth} that proves semi-smoothness of  $u\mapsto\max(a,u)$ from $L^q(\Omega)$ to $ L^p(\Omega)$ even if $a\in L^r(\Omega)$, with $1\leq p\leq r<q\leq \infty$.
In Section \ref{sec:NewtonNonRegNEP} we apply the semi-smooth Newton method to the augmented NEP \eqref{eq:intro2}, state a convergence result and give a detailed description of the implementation applying a finite element discretization. The equivalence of the semi-smooth Newton method and the active-set method is treated in Section \ref{sec:activeSet}. To illustrate our theoretical findings and to compare the two presented methods we study numerical examples in detail.

\section{The Non-Reducible Problem}
In this section we state the problem setting, establish optimality conditions and give a sufficient condition that yields existence of unique solutions for a non-reducible NEP.
\subsection{Problem Setting} \label{sec:setting}
Let $\Omega\subset \R^n, n\in\lbrace 1,2,3\rbrace$. Let us first consider the case if each player aims at solving the following Nash equilibrium problem in the optimal control setting with identical tracking type cost functional for each player, i.e.,
\begin{align*}
\min\limits_{u^\nu \in L^2(\Omega)} \; &\frac{1}{2}\norm{Su-y_d^\nu}_{L^2(\Omega)}^2 +\frac{\alpha}{2}\norm{u^\nu}_{L^2(\Omega)}^2\\
 \text{s.t. }\quad & u^\nu\in\Uadnu ,
\end{align*}
where
\[
\Uadnu:=\left\lbrace u^\nu\in L^2(\Omega)\colon u_a^\nu(x)\leq u^\nu(x)\leq u_b^\nu(x) \right\rbrace 
\]
with $u_a^\nu,u_b^\nu\in L^2(\Omega)$. Clearly the set $\Uadnu$ is bounded and convex, hence weakly compact. The operator $S\colon H^{-1}(\Omega)^N\rightarrow Y\hookrightarrow L^q(\Omega)^N$ denotes the solution operator of a linear elliptic partial differential equation and the state space $Y$ is assumed to be embedded in $L^q(\Omega)^N$ with $q>2$. For instance we can consider $S$ as the control-to-state map of
\begin{equation}\label{eq:pdeDirichlet}
\begin{alignedat}{2}
Ay &= \sum_{\nu=1}^N u^\nu \quad &&\text{in } \Omega,\\
y  &= 0 						 &&\text{on }\partial \Omega.
\end{alignedat}
\end{equation}
We assume that the operator $A \colon Y\rightarrow H^{-1}(\Omega)$ is linear, bounded and continuously invertible. Note that this is the case for $A:=-\Delta$, which is an isomorphism from $H_0^1(\Omega)$ to its dual $H^{-1}(\Omega)$. This can be proven using the Lax-Milgram theorem. In this setting it is convenient to take $Y := H_0^1(\Omega) \cap \CO$, see e.g. \cite{Casas2012}. Since $u^\nu\in L^2(\Omega)\hookrightarrow H^{-1}(\Omega)$ the state equation is well-posed. The corresponding solution operator satisfies 
\[
S \colon u\mapsto y=A^{-1}\sum_{\nu=1}^Nu^\nu,\quad S\colon H^{-1}(\Omega)^N\rightarrow H_0^1(\Omega)\cap \CO.
\]
Since for $n=1$ we have the embedding $H_0^1(\Omega) \hookrightarrow C(\bar \Omega)$, for $n=2$ we have $H^1(\Omega)\hookrightarrow L^q(\Omega)$ with $ 1\leq q <\infty$ and for $n=3$ we still have $H^1(\Omega)\hookrightarrow L^6(\Omega)$, hence the required assumption on $S$ is satisfied.
Due to the linearity of $A^{-1}$ we have
\begin{align*}
Su = \sum_{\nu=1}^N A^{-1}u^\nu := \sum_{\nu=1}^N S_\nu u^\nu, \quad S_\nu\colon H^{-1}(\Omega) \rightarrow Y\hookrightarrow L^q(\Omega), \; S_\nu u^\nu := A^{-1} u^\nu. 
\end{align*}
Problems of this type can be reduced to a single convex control problem given by 
\begin{align*}
\underset{u\in L^2(\Omega)}{\min} \; &\frac{1}{2}\norm{Su}^2_{L^2(\Omega)} + \sum_{\nu=1}^N \left( -(u^\nu,S_\nu^* y_d^\nu) + \frac{\alpha}{2}\norm{u^\nu}^2_{L^2(\Omega)}\right)\\
\text{s.t.}\quad &u\in\Uad.
\end{align*}
Here $\Uad := \Uad^1 \times \dots \times \Uad^N$. This easily yields the existence of a unique equilibrium for $\alpha>0$ \cite[Proposition 3.10]{HintermSurowiecKaemm2015gnep}. Let us now investigate the case if the tracking type functional for the $\nu$-th player is considered on a subset $\Omega_\nu\subseteq \Omega$ only. In this case reduction to a single control problem is not possible and we will refer to this type of problem as a \emph{non-reducible NEP}. We consider the cost functional 
\begin{align*}
f_\nu(u):= \frac{1}{2}\norm{Su-y_d^\nu}_{L^2(\Omega_\nu)}^2 +\frac{\alpha}{2}\norm{u^\nu}_{L^2(\Omega)}^2
\end{align*}
and analyze the Nash equilibrium problem
\begin{align}
\underset{u^\nu\in L^2(\Omega)}{\min} f_\nu(u)\quad  \text{ s.t. } u^\nu\in \Uadnu.
\tag{$P_\nu$}
\label{eq:prob_non_red}
\end{align}
Our aim is to study under which conditions \eqref{eq:prob_non_red} admits a unique solution. Let $\Omega \subseteq \R^n$ be a bounded Lipschitz domain and $\Omega_\nu \subseteq \Omega$ for $\nu=1,...,N$. For further use we define the characteristic function
\[
\chi_\nu(x)\colon \Omega \to \R, \quad x \mapsto \begin{cases}
1 & \text{if } x \in \Omega_\nu,\\0 & \text{else},
\end{cases}
\]
as well as $U:= L^2(\Omega)^N$ and the operator
\begin{align}\label{eq:defF}
F(u)\colon U \to U, \quad F^\nu (u) &:= D_{u^\nu}f_\nu(u) = S_\nu^\ast \chi_\nu (Su - y_d^\nu) + \alpha u^\nu,
\end{align}
where $D_{u^\nu}$ denotes the partial G\^ateaux derivative with respect $u^\nu$. Due to the convexity of the cost functional solutions of the NEP can be characterized
via controls $u\in U$ that solve the variational inequality
\begin{align}\label{eq:fonc_pert_NEP}
& (F(\bu),v-\bu)_U \geq 0, \qquad \forall v\in \Uad \notag\\
\Leftrightarrow\ &\sum_{\nu=1}^N \big( S_\nu^\ast \chi_\nu (S\bu - y_d^\nu) + \alpha \bu^\nu  ,v^\nu-\bu^\nu \big)\geq 0, \qquad \forall v^\nu\in\Uadnu
\end{align}

We will exploit this relation to prove uniqueness of solutions of problem \eqref{eq:prob_non_red}. It is well known that \eqref{eq:fonc_pert_NEP} can be equivalently formulated using the projection operator $P_\Uad$ onto the set $\Uad$. A solution $\bar u \in \Uad$ of \eqref{eq:prob_non_red} can be characterized by the equation
\begin{equation}\label{eq:fixed_point_formulation}
0 = \bar u - P_\Uad \big( \bar u - \gamma F(\bar u) \big)
\end{equation}
for all $\gamma > 0$. Furthermore, this formulation allows us to tackle the problem using a semi-smooth Newton method.

\subsection{Existence and Uniqueness of Solutions}
If the variational inequality \eqref{eq:fonc_pert_NEP} is uniquely solvable the NEP \eqref{eq:prob_non_red} admits a unique solution. It is well known that this is the case if $F$ is strongly monotone \cite[Theorem 1.4]{KinderlStampacchia2000varinequalities}. The next theorem states that this is the case if the regularization parameter $\alpha$ is chosen large enough, depending on the sets $\Omega_\nu$. Let us define the set 
\[
Z := \bigcup\limits_{\nu = 1}^N \Omega_\nu
\]
with associated characteristic function $\chi_Z$. In order to deal with the different sets $\Omega_\nu$ we need the following assumption.

\begin{assumption}\label{ass:alpha}
Assume that the regularization parameter $\alpha$ satisfies the inequality
\begin{equation}\label{eq:ass_alpha}
	\alpha > \frac{1}{4} \sum\limits_{\nu=1}^N \|  \chi_Z(  S_\nu - \chi_\nu S_\nu )   \|^2_{L^2(\Omega) \to L^2(\Omega)}.
\end{equation}
We will refer to the right hand side of \eqref{eq:ass_alpha} as the \emph{offset}.
\end{assumption}
Note that for a fixed operator $S$ the offset depends only on the sets $\Omega_\nu$. Before we use this assumption to prove existence and uniqueness of solutions let us analyze this condition.

Let us assume that $S_\nu\colon H^{-1}(\Omega) \to  H_0^1(\Omega) \cap C(\bar \Omega)$. As already mentioned this is the case for the operator $S_\nu = (-\Delta)^{-1}$. It is well known tat the solution operator $S_\nu$ is continuous. Hence, we know that the number
\[
C := \max\limits_{\nu=1,...,N}\sup\limits_{\|w\|_{L^2(\Omega)} = 1} \|S_\nu w\|_{L^\infty(\Omega)}< \infty
\]
exists. Now we obtain
\begin{align*}
\|\chi_Z(S_\nu - \chi_\nu S_\nu)\|_{L^2(\Omega) \to L^2(\Omega)} &= \sup\limits_{\|w\|_{L^2(\Omega)} = 1} \|\chi_Z (S_\nu - \chi_\nu S_\nu)w\|_{L^2(\Omega)}\\
& = \sup\limits_{\|w\|_{L^2(\Omega)} = 1}  \left( \int_Z (1 - \chi_\nu)^2 (S_\nu w)^2  \dx \right)^{\frac{1}{2}}\\
&\leq \sup\limits_{\|w\|_{L^2(\Omega)} = 1}  \norm{S_\nu w}_{L^\infty(\Omega)} \left(\int_Z (1 - \chi_\nu)^2 \dx \right)^{\frac{1}{2}}\\
& \leq C \sqrt{  \text{meas}(Z \setminus \Omega_\nu) }.
\end{align*}
Hence,
\[
\sum\limits_{\nu=1}^N \|  \chi_Z(  S_\nu - \chi_\nu S_\nu )   \|^2_{L^2(\Omega) \to L^2(\Omega)} \leq C^2 \sum\limits_{\nu=1}^N   \text{meas}(Z \setminus \Omega_\nu).
\]
Thus, we can interpret the offset from \eqref{eq:ass_alpha} as the maximum difference of the set $Z = \cup_\nu \Omega_\nu$ and the sets $\Omega_\nu$.
If $\Omega_\nu =  \Omega$ for all $\nu$ this offset is obviously zero and we are in the setting of a reducible NEP. However, if the offset is too large, the existence of minimizers can not be guaranteed by our theory for all $\alpha > 0$. 
Let us now start to exploit Assumption \ref{ass:alpha}.
\begin{theorem}\label{thm:NEPunique}
Let Assumption \ref{ass:alpha} be satisfied. Then there exists an unique solution of the non-reducible NEP \eqref{eq:prob_non_red}. 
\end{theorem}
\begin{proof}
As already mentioned it is enough to show that the operator $F$ defined in \eqref{eq:defF} is strongly monotone. A calculation reveals for arbitrary $u,v \in U$
\begin{align*}
\left( F(u) - F(v), u -v \right)_U &=\sum\limits_{\nu=1}^N \left(  S_\nu^\ast \chi_\nu (Su - y_d^\nu)  + \alpha u^\nu - S_\nu^\ast \chi_\nu (Sv - y_d^\nu)  - \alpha v^\nu, u^\nu - v^\nu  \right)_{L^2(\Omega)}\\
&= \sum\limits_{\nu=1}^N \left(   Su - Sv, \chi_\nu S_\nu (u^\nu - v^\nu)   \right)_{L^2(\Omega)} + \alpha \|u-v\|_U^2\\
&= \sum\limits_{\nu=1}^N \left(  ( Su - Sv, \chi_\nu S_\nu (u^\nu - v^\nu)   \right)_{L^2(Z)} + \alpha \|u-v\|_U^2.
\end{align*}
We now use the decomposition
\[
\sum\limits_{\nu=1}^N \chi_\nu S_\nu = \sum\limits_{\nu=1}^N S_\nu - \sum\limits_{\nu=1}^N \left( S_\nu - \chi_\nu S_\nu   \right)
\]
and Young's inequality to obtain the following estimate
\begin{align*}
\bigg( &F(u) - F(v), u -v \bigg)_U\\
&= \|Su  - Sv\|_{L^2(Z)}^2 - \left( Su - Sv,   \sum\limits_{\nu=1}^N \left( S_\nu - \chi_\nu S_\nu   \right)(u^\nu - v^\nu)  \right)_{L^2(Z)} + \alpha \|u-v\|_U^2\\
&\geq  -  \frac{1}{4}  \left\| \sum\limits_{\nu=1}^N \chi_Z \left( S_\nu - \chi_\nu S_\nu   \right)(u^\nu - v^\nu) \right\|_{L^2(\Omega)}^2 + \alpha \|u-v\|_U^2\\
&\geq - \frac{1}{4}\left(  \sum\limits_{\nu=1}^N \|  \chi_Z(  S_\nu - \chi_\nu S_\nu )   \|^2_{L^2(\Omega) \to L^2(\Omega)}  \right) \left(   \sum\limits_{\nu=1}^N \| u^\nu - v^\nu \|_{L^2(\Omega)}^2   \right)  + \alpha \|u-v\|_U^2\\
&= \left(   \alpha - \frac{1}{4} \sum\limits_{\nu=1}^N \|  \chi_Z(  S_\nu - \chi_\nu S_\nu )   \|^2_{L^2(\Omega) \to L^2(\Omega)}   \right) \|u-v\|_U^2.
\end{align*}
Due to our assumption on $\alpha$ we now conclude that the operator $F$ is strongly monotone.
\end{proof}
\begin{remark}
For getting existence of a solution of the non-reducible NEP \eqref{eq:prob_non_red} without requiring its uniqueness, it would be enough to claim monotonicity of $F$ only, see \cite[Theorem 1.4]{KinderlStampacchia2000varinequalities}. Thus, we can relax Assumption \ref{ass:alpha} slightly, by assuming that $\alpha$ satisfies 
\[ 	\alpha \geq \frac{1}{4} \sum\limits_{\nu=1}^N \|  \chi_Z(  S_\nu - \chi_\nu S_\nu )   \|^2_{L^2(\Omega) \to L^2(\Omega)}.\]
\end{remark}
The condition on the regularization parameter $\alpha$ is needed to guarantee the existence of a unique solution of \eqref{eq:prob_non_red}. If $\alpha$ is chosen too small the resulting operator $F$ might not be strongly monotone.
It is quite interesting that for $\alpha > 0$ the operator $F$ is still strongly monotone if all the domains $\Omega_\nu$ coincide, but not necessarily equal to the domain $\Omega$. 

\begin{corollary}
If $\Omega_\nu = \Omega_0 \subseteq \Omega$ for all $\nu =1,...,N$, then the NEP  \eqref{eq:prob_non_red} is uniquely solvable for all $\alpha > 0$.
\end{corollary}

\subsection{The Augmented NEP}\label{sec:subsec:augNEP}
In this section we want to extend our result from the former section. Let $\Omega \subseteq \R^n$ with $n\in\lbrace 1,2,3\rbrace$ be a bounded Lipschitz domain. Problems of this type are arise during the process of solving generalized Nash equilibrium problems (GNEPs) where the individual problem is given by
\begin{equation}
\begin{alignedat}{3}
\underset{u^\nu\in  L^2(\Omega)}\min\;  &\ f_\nu(u^\nu):=\frac{1}{2}\norm{Su-y_d^\nu}^2_{L^2(\Omega_\nu)}+\frac{\alpha}{2}\norm{u^\nu}^2_{L^2(\Omega)}\ \\
\text{s.t.} \quad & u^\nu\in \Uadnu\\
&(Su)(x)\leq \psi(x) \quad \text{a.e. in }\bar\Omega,
\end{alignedat}
\label{eq:GNEP}
\end{equation}
by applying an augmented Lagrange method, see \cite{KanzowKarlSteckWachsmuth2017gnep}. Here, $\psi \in C(\bar \Omega)$ defines an additional upper bound for the state $y = Su$.  To guarantee the existence of Lagrange multipliers it is necessary to have $\psi \in C(\bar \Omega)$, see \cite{KanzowSteckWachsmuth2018augmented}. Solving \eqref{eq:GNEP} with an augmented Lagrange method requires a sequence of solutions of the following Nash equilibrium problem, where $\alpha, \rho>0$
\begin{equation}
\begin{alignedat}{5}
\min\limits_{u^\nu \in L^2(\Omega)} \;& f_\nu^{AL}(u,\mu,\rho):=\frac{1}{2}\norm{Su-y_d^\nu}_{L^2(\Omega_\nu)}^2 +\frac{\alpha}{2}\norm{u^\nu}_{L^2(\Omega)}^2 + \frac{1}{2\rho} \|\bmu(u)\|_{L^2(\Omega)}^2\\
\text{s.t.} \quad &u^\nu\in \Uadnu
\end{alignedat}
\tag{$P_{\nu}^{AL}$}
\label{eq:augNEP_offset}
\end{equation}

with $\bar \mu(u):=(\mu+\rho(Su-\psi))_+$ and $\mu\geq 0$ is assumed to be a function in $L^2(\Omega)$. 
Defining 
\begin{align*}
F(u):=\left(D_{u^1}f^{AL}_1(u,\mu,\rho),\dots,D_{u^N}f^{AL}_N (u,\mu,\rho) \right)
\end{align*}
it is again the convexity of the cost functional that allows us to characterize the solution of the NEP 
via controls $\bu\in U$ that solve the variational inequality
\begin{align*}
& (F(\bu),v-\bu)_U \geq 0, \qquad \forall v\in \Uad \notag\\
\Leftrightarrow\ &\sum_{\nu=1}^N( S_\nu^\ast (\chi_\nu (S\bu - y_d^\nu))+\bmu(u)) + \alpha \bu^\nu  ,v^\nu-\bu^\nu)\geq 0, \qquad \forall v^\nu\in\Uadnu.
\end{align*}
Similar as above an equivalent formulation using the projection operator can be established, see \eqref{eq:fixed_point_formulation}.
From Theorem \ref{thm:NEPunique} we know that the mapping
\[
u\mapsto \left(D_{u^1} f_1(u^1),\dots,D_{u^N} f_N(u^N)\right)
\]
 is strongly monotone if Assumption \ref{ass:alpha} is satisfied. Furthermore, we know that the function
 \[
 u \mapsto \frac{1}{2\rho}\norm{\bar \mu(u)}_{L^2(\Omega)}^2
 \]
 is convex and its derivative is monotone. Hence $F$ is strongly monotone and Theorem \ref{thm:NEPunique} can easily be adapted to that case.

\begin{theorem}\label{thm:augNEP_unique}
Let Assumption \ref{ass:alpha} be satisfied. Then there exists an unique solution of problem \eqref{eq:augNEP_offset}. Further, if $\Omega_\nu = \Omega_0 \subseteq \Omega$ for all $\nu =1,...,N$, then the NEP  \eqref{eq:augNEP_offset} is uniquely solvable for all $\alpha > 0$.
\end{theorem} 

We will deepen our studies of problem \eqref{eq:augNEP_offset} in Section \ref{sec:NewtonNonRegNEP}. Here, we will among others derive the corresponding Newton iteration that allows us to solve the problem numerically with superlinear convergence.

\section{Semi-smooth Newton Method}\label{sec:SSNMgen}
This section aims at collecting important notations and results from literature in order to introduce the semi-smooth Newton method and state the well known theorem that yields superlinear convergence of just this method. We complete this section by contributing Lemma \ref{lemma:max_semismooth} that proves semi-smoothness of  $u\mapsto\max(a,u)$ from $L^q(\Omega)$ to $ L^p(\Omega)$ even if $a\in L^r(\Omega)$, with $1\leq p\leq r<q\leq \infty$.\medskip\newline
To simplify our notation we define 
\[
Y_q := L^q(\Omega)^N,
\]
for some $q \in [1, \infty)$. Recall that $U =L^2(\Omega)^N$, hence we have $Y_q \hookrightarrow U$ for $q \geq 2$.
We want to apply Newton's method to an equation similar to \eqref{eq:fixed_point_formulation}. Note that due to the regularization term we can always reformulate our necessary optimality condition to
\[
0 = \bar  u - P_\Uad \big(  \bu - \gamma(  \tilde F + \alpha \bar u  )  \big)
\]	
with a function $\tilde F$. From now on we always set $\gamma := \alpha^{-1} > 0$ to simplify our equation. Hence, we are interested in finding zeros of functions $G:U\rightarrow U$ defined as
\begin{align}
G(u):=u-P_{\Uad}\left(-\frac{1}{\alpha} \tilde F(u)\right). \label{eq:defG}
\end{align}

We make the following assumption on $\tilde F$ in order to be able to apply the semi-smooth Newton method, which is introduced in the next section, see Definition \ref{def:Newton_derivative}.
\begin{assumption}
We assume that $\tilde F$ from \eqref{eq:defG} satisfies $\tilde F:U\rightarrow Y_q$ with $q > 2$ such that each component $\tilde F^\nu$ is semi-smooth and locally Lipschitz from $U$ to $L^q(\Omega)$ for all $u$.
\label{ass:Fpart1}
\end{assumption}

\subsection{The Semi-Smooth Newton Method}

Applying semi-smooth Newton methods requires the notion of semi-smoothness or Newton-differentiable functions. In this chapter let $U$ denote an arbitrary Banach space.
\begin{definition}[\bf{Newton derivative}]\label{def:Newton_derivative}
Let $U,Y$ be Banach spaces. The mapping $G\colon U \rightarrow Y$ is called Newton differentiable or semi-smooth if there exists a linear and continuous mapping $D_N G\colon U \rightarrow \mathcal{L}(U, Y)$ such that
\begin{equation}\label{eq:Newton_derivative}
\norm{G(u+v)-G(u)- D_N G(u+v)v}_Y=o(\|v\|_U)
\end{equation}
for every $u\in U$. The mapping $D_N G$ is called the \emph{Newton derivative} of $G$.
\end{definition}
Let us present a well-known class of functions which are semi-smooth. 
\begin{lemma}\label{lemma:Frechet-semismooth}
Let $U,Y$ denote Banach spaces. Every Fr\'echet differentiable function $F\colon U\rightarrow Y$ with continuous Fr\'echet derivative $F'$ is Newton differentiable with Newton derivative $F'$.
\end{lemma}

The semi-smooth Newton method for finding a solution of $G(u) = 0$ is given in the following algorithm.
\begin{algorithm}[H]
    \caption{Semi-smooth Newton method}
    \label{alg:semi-smooth-gen}
    \begin{algorithmic}
     \item[] Choose $u_0\in U$
     \item[   ] For $k=0,1,2,...$ repeat:
        \item[1:] Compute $\delta_k$ by solving \[D_N G(u_k)\delta_k=-G(u_k).\]
        \item [2:] Set $u_{k+1}:=u_k+\delta_k$.
    \end{algorithmic}
\end{algorithm}

Let $\bar u$ solve $G(u)=0$, where $G\colon U\rightarrow Y$. Let us assume that the mappings $D_N G(u_k) \in \mathcal{L}(U,Y)$ are invertible. Applying the definition of a Newton step we get
\begin{align*}
%
\norm{u_{k+1}-\bar{u}}_U \leq \norm{D_N G(u_k) ^{-1}}_{Y\rightarrow U}\norm{D_N G(u_k) (u_k-\bar{u})-\left(G(u_k)-G(\bar{u})\right)}_Y.
\end{align*}
Based on this estimate it is well known and easy to see (\cite[Theorem 8.16]{ItoKunisch2008LagMult}), that Algorithm \ref{alg:semi-smooth-gen} converges superlinearly to a solution $\bar{u}$ of $G(u) = 0$ if the following two conditions hold:
\begin{enumerate}
\item \emph{Approximation condition}: The mapping $G\colon U\rightarrow Y$ is Newton-differentiable at $\bu$ with Newton derivative $D_NG : U\rightarrow \mathcal{L}(U,Y)$.
\item \emph{Regularity condition}: There exists a constant $c>0$ and an $\varepsilon>0$ such that for every $u$ satisfying $\norm{u-\bar{u}}_U\leq\varepsilon$ all $D_N G(u)$ are invertible and $\norm{D_N G(u)^{-1}}_{Y\rightarrow U} \leq c$ holds.
\end{enumerate}

To end this section we will recall some properties of Newton differentiable functions, see \cite[Theorem 2.10]{HinzePinnauUlbrich2009optimization}.

\begin{lemma}\label{lemma:propNewFunc}
Let $U,Y,Z,U_i,Y_i$ be Banach spaces. 
\begin{itemize}[noitemsep]
\item[a)]If the operators $G_i\colon U\rightarrow Y_i$ are Newton differentiable at $u$ then $(G_1,G_2)$ is Newton differentiable at $u$.
\item[b)]If $G_i\colon U\rightarrow Y, i=1,2$ are Newton differentiable at $u$ then $G_1+G_2$ is Newton differentiable at $u$ with Newton derivate $(D_N G_1 + D_N G_2)$.
\item[c)] Let $G_1:Y\rightarrow Z$ and $G_2:U\rightarrow Y$ be Newton differentiable at $G_2(u)$ and $u$, respectively. Assume that $D_N G_1$ is bounded near $y=G_2(u)$ and that $G_2$ is Lipschitz-continuous near $u$. Then $G=G_1\circ G_2$ is Newton differentiable with
\[ D_N G(u)=  M_1M_2 \quad \text{with}  \quad M_1:=D_N G_1(G_2(u)), \; M_2:=D_N G_2(u).\]
\end{itemize}
\end{lemma}

\subsection{Semi-Smoothness of the Projection Operator}
For our later application we will need semi-smoothness of the mapping 
\[
u\mapsto (\mu+\rho(Su-\psi))_+,
\]
see Section \ref{sec:subsec:NewtonIt}. Since $\mu$ is only a $L^2(\Omega)$ function we cannot expect from the known result \cite[Theorem 4.4]{MUlbrich2011SSN} that the mapping
\[
\max(0,\mu+\rho(Su-\psi))=\mu-\rho\psi + \max(-\mu+\rho\psi,Su)
\]
is semi-smooth from $L^q(\Omega)$ to $L^2(\Omega)$.
In \cite[Example 8.12]{ItoKunisch2008LagMult} Ito and Kunisch investigated the semi-smoothness of superposition operators 
$$F\colon L^q(\Omega)\rightarrow L^p(\Omega),\quad F(u)(x)=f(u(x)) \text{ for a.e. } x\in \Omega,$$ where $1\leq p<q\leq \infty$ and $f:\R\rightarrow \R$ is semi-smooth and globally Lipschitz continuous. However, due to the dependence of $a$ and $b$ on the $x$-variable the mapping $u\mapsto \max(a,\min(u,b))$ cannot be built via superposition. Nevertheless, since the regularity of the functions $a$ and $b$ isn't needed in the proof one can apply similar arguments. 
\begin{theorem}\label{theo:maxminsemismooth}
Let $a,b \in L^r(\Omega)$ with $a \leq b$ and $1 \leq p\leq r < q \leq \infty$. The mapping $m: L^q(\Omega) \to L^p(\Omega), \; u \mapsto \max( a, \min(u,b)  )$ is semi-smooth with Newton derivative
\begin{equation}\label{eq:proj_newton}
L^s(\Omega) \ni h(u)(x) = \begin{cases}
0 & \text{if} \quad u(x) \geq b(x),\\
1 & \text{if} \quad u(x) \in (a(x), b(x)),\\
0 & \text{if} \quad u(x) \leq a(x),
\end{cases}
\end{equation}
where $s$ is chosen such that $\frac{1}{p} = \frac{1}{s} + \frac{1}{q}$ holds.
\end{theorem}

\begin{proof}
A similar proof can be found in the PhD-Thesis \cite{SteckPhD2018}. Let $u \in L^q(\Omega)$ be arbitrary and $(s_k)_k \subset L^q(\Omega)$ be a (strong) nullsequence. Furthermore, define $u_k := u + s_k$ and $d_k := h(u_k)$. We have to check property \eqref{eq:Newton_derivative}.

First we extract a subsequence $(s_k)_{k \in I}$ with an index set $I$ such that $s_k(x) \to_I 0$ for almost all $x \in \Omega$. To shorten the notation we furthermore define $v := m(u)$ and $v_k := m(u_k)$. We now use
It is known \cite[Example 2.5]{HinzePinnauUlbrich2009optimization} that the mapping $\tilde{m}: \R \to \R, \; x \mapsto \max( a, \min(x,b)  )$ with $a,b\in\R$ is semi-smooth.
%
Hence, we obtain
\[
s_k(x)^{-1} (  v_k(x) - v(x) - d_k(x) s_k(x)  ) \to_I 0
\]
for almost all $x \in \Omega$. The quotient on the left side is understood to be zero whenever $s_k(x)=0$. Now we use that the projection $m$ is nonexpansive and obtain
\begin{align*}
|  v_k(x) - v(x) - d_k(x) s_k(x)  | & \leq |  v_k(x) - v(x) | +| d_k(x) s_k(x)  |\\
&\leq |u(x) + s_k(x) - u(x)| + |s_k(x)|\\
&\leq 2 |s_k(x)|.
\end{align*}
By applying  Lebesgue’s dominated convergence theorem we obtain 
\[
s_k^{-1} (  v_k - v - d_k s_k  ) \to_I 0
\]
in $L^r(\Omega)$ for all $r \in [1, \infty)$. Hence, by applying Hölder's inequality we get with $\frac{1}{p} = \frac{1}{s} + \frac{1}{q}$
\begin{equation}\label{eq:proj_proof_1}
\frac{\|  v_k - v - d_k s_k  \|_{L^p(\Omega)}}{\|s_k\|_{L^q(\Omega)}} \leq \|s_k^{-1} (  v_k - v - d_k s_k  ) \|_{L^s(\Omega)} \to_I 0.
\end{equation}
Since this argumentation can be repeated for any subsequence of $(s_k)_k$ the limit in \eqref{eq:proj_proof_1} holds in fact for the whole sequence.
\end{proof}
 
In the same manner we obtain the following result.
\begin{lemma}\label{lemma:max_semismooth}
Let $a \in L^r(\Omega)$ and $1 \leq p \leq r< q \leq \infty$. The mapping $m: L^q(\Omega) \to L^p(\Omega), \; u \mapsto \max( a, u )$ is semi-smooth with Newton derivative
\begin{equation}\label{eq:proj_max}
L^s(\Omega) \ni h(u)(x) = \begin{cases}
1 & \text{if} \quad u(x) > a(x),\\
0 & \text{if} \quad u(x) \leq a(x),
\end{cases}
\end{equation}
where $s$ is chosen such that $\frac{1}{p} = \frac{1}{s} + \frac{1}{q}$ holds.
\end{lemma}

Note that the norm gap $p < q$ is indispensable for Newton differentiability of the projection operator, see for instance \cite[Example 8.14]{ItoKunisch2008LagMult}. Hence, the functions defined in \eqref{eq:proj_newton} and \eqref{eq:proj_max} can in general not serve as a Newton derivative for $m\colon L^2(\Omega)\rightarrow L^2(\Omega)$, see \cite[Proposition 4.1]{HintermItoKunisch2002primalDualActiveSetAsSSN}.
This causes trouble proving superlinear convergence since we cannot expect that the approximation condition holds. To bridge this norm gap, one needs additional structure. For problems that involve partial differential equations this structure is often given by smoothing properties of the corresponding solution operators.
To finish, let us briefly comment on the semi-smoothness of the projection operator $P_{\Uad} : L^q(\Omega)^N \to L^2(\Omega)^N$.
 The mapping
\[
\Pi_\nu\colon L^q(\Omega)^N \to L^q(\Omega), \; u \mapsto u^\nu.
\]
is linear and Fr\'echet differentiable, hence semi-smooth by Lemma \ref{lemma:Frechet-semismooth}.  Applying the chain rule (Lemma \ref{lemma:propNewFunc} c)) we now obtain that
\[
P_{\Uad}^\nu(u) = \min(\max(a^\nu,  \Pi_\nu(u)  ) ,b^\nu) 
\]
is a composition of semi-smooth functions, hence semi-smooth from $L^q(\Omega)^N \to L^2(\Omega)$, see 
Lemma \ref{theo:maxminsemismooth}.
Using Lemma \ref{lemma:propNewFunc} a) we obtain that $P_{\Uad}$ is semi-smooth from $L^q(\Omega)^N \to L^2(\Omega)^N$.

\subsection{Convergence Analysis}

To simplify our notation let us introduce the following notation. Let $d \in U$ with components $d^\nu \in L^2(\Omega)$ and  $M \in \mathcal{L}(U,U)$. We define the product $d \cdot M =dM \in \mathcal{L}(U,U)$ in a component-wise manner
\begin{equation}\label{eq:diag_def}
\left( d \cdot M (u) \right)^\nu:= d^\nu M^\nu(u) \in L^2(\Omega).
\end{equation}
Hence, $ d \cdot M: U \to U$. In a similar way we define $d \cdot u \in U$ for some $u \in U$. The chain rule from Lemma \ref{lemma:propNewFunc} c) allows us to show semi-smoothness of $G$ from \eqref{eq:defG}.

\begin{lemma}\label{lem:chainrule}
Let Assumption \ref{ass:Fpart1} be satisfied. Then the operator $G:U\rightarrow U$ is Newton differentiable with Newton derivative
\begin{align*}
D_N G(u) =  \mathrm{Id} + \frac{1}{\alpha} \chi_{\Iset}(u) D_N \tilde{F}(u)\in \mathcal{L}(U,U) , 
\end{align*}
where the components of $\chi_I(u)$ are given as
\begin{align}
(\chi_{\Iset}(u))^\nu(x)
:= \begin{cases}
0 & \text{if} \quad - \frac{1}{\alpha} \tilde F(u)^\nu (x) \geq b^\nu(x),\\
1 & \text{if} \quad - \frac{1}{\alpha} \tilde F(u)^\nu (x) \in (a^\nu(x), b^\nu(x) ),\\
0 & \text{if} \quad - \frac{1}{\alpha} \tilde F(u)^\nu (x) \leq a^\nu(x),\\
\end{cases} 
\label{eq:defd}
\end{align}
for almost all $x\in \Omega$.
\end{lemma}

\begin{proof}This follows directly by Assumption \ref{ass:Fpart1}, the chain rule (Lemma \ref{lemma:propNewFunc} c)) and Lemma \ref{lemma:Frechet-semismooth}. Further, the representation of the derivative follows immediately with the chain rule where the derivative of $P_{\Uad}$ can be deduced from \cite[Example 2.5]{HinzePinnauUlbrich2009optimization}.
\end{proof}
Due to Assumption \ref{ass:Fpart1} the operator $G$ from \eqref{eq:defG} already satisfies the approximation condition. It remains to check on the regularity condition. For a given iterate $u_k$ let us define $\chi_{\Iset_k} := \chi_{\Iset}(u_k) \in U$. Let us now consider the bilinear form 
\begin{align}\label{eq:bilinform}
 a(w,v):=\left(\left(\mathrm{Id}+\frac{1}{\alpha}\chi_{\Iset_k} D_N \tilde{F}(u_k) \chi_{\Iset_k}\right)w,v\right)_U. 
\end{align}
We make the following assumption.
\begin{assumption}\label{ass:bilinform}
Assume that the bilinear form \eqref{eq:bilinform} is coercive for all $u_k \in U$, i.e. there exists a constant $c>0$ such that for all $w \in U$ it holds $a(w,w)\geq c\norm{w}^2_U$.
\end{assumption}

In fact, Assumption \ref{ass:bilinform} is obviously satisfied if $D_N \tilde{F}(u_k) $ is positive semidefinite with respect to the scalar product in $U$, i.e.,
\[
(D_N \tilde{F}(u_k)  w,w)_U \geq 0 \qquad \forall w\in U
\]
holds. Furthermore, it is well known \cite[Proposition 4.1.6]{Burachik2008Setvalued} that if $\tilde{F}:U\rightarrow Y_q$ is G\^ateaux differentiable for all $u\in U$ and monotone, then the G\^ateaux derivative $D\tilde{F}$ is positive semidefinite for every $u$.
In the next section we will explicitly show that the needed assumptions for superlinear convergence are satisfied for our NEP. Furthermore, please note that the structure of the first part of the bilinear form is very similar to the structure of the Newton derivative of $G$. However, the additional characteristic function available in the bilinear form allows us to prove superlinear convergence. This is part of the next theorem.

\begin{theorem} \label{theo:convSSNgen} 
Let $\bar{u}$ solve $G(u)=0$ with $G$ as given in \eqref{eq:defG}. Let Assumption \ref{ass:Fpart1} and \ref{ass:bilinform} be satisfied. If $\norm{u_0-\bar{u}}_U$ is sufficiently small the iterates $u_k$ from Algorithm \ref{alg:semi-smooth-gen} converge superlinear to $\bar u$.
\end{theorem}
\begin{proof}
The proof uses standard arguments for semi-smooth Newton methods. For the readers convenience it can therefore be found in the appendix.
\end{proof}

\section{Newton Iteration for the Non-Reducible NEP} \label{sec:NewtonNonRegNEP}

We now want to study the semi-smooth Newton method applied to problem \eqref{eq:augNEP_offset}.

\subsection{Newton Iteration and Convergence Result} \label{sec:subsec:NewtonIt}

We aim at solving
\begin{align}\label{eq:GforNewtonit}
G(u):= u-P_{\Uad}\left(-\frac{1}{\alpha}\tilde F(u)\right) =0,
\end{align} 
where $\tilde{F}:U\rightarrow Y_q$ and the $\nu$-th component $\tilde{F}^\nu$ is given as the adjoint state
\[
\tilde F^\nu(u) = p^\nu(u):=  S_\nu^\ast \left(\chi_\nu (Su - y_d^\nu) +\bmu(u)\right). 
\]

\begin{lemma}
The operator $
u \mapsto  \left(  S_\nu^\ast \left(\chi_\nu (Su - y_d^\nu)\right)+\bmu(u) \right)_{\nu=1}^N
$
satisfies Assumption \ref{ass:Fpart1}.
\end{lemma}
\begin{proof}
We note that $S_\nu^\ast : H^{-1}(\Omega) \to (H^{-1})^\ast = H_0^1(\Omega)$. 
Due to embedding theorems we obtain that
%
%
$p^\nu(u)= S_\nu^*(\chi_\nu (Su-y_d^\nu)+(\mu+\rho(Su-\psi))_+)$ maps from $U$ to $L^q(\Omega)$ with some $q>2$. 
Splitting the adjoint state in two parts
\[
p^\nu(u):= S_\nu^*\chi_\nu(Su-y_d^\nu)+S_\nu^*(\mu+\rho(Su-\psi))_+. 
\]
we see clearly that the first part is continuously Fr\'echet differentiable, hence semi-smooth due to Lemma \ref{lemma:Frechet-semismooth} and Lipschitz continuous from $U$ to $L^q(\Omega)$. Recall that we need the norm gap in order to prove semi-smoothness of the projection operator. For the second part, we know from Lemma \ref{lemma:max_semismooth} and the regularity conditions on $S$ that the mapping $u\mapsto \max(0,\mu+\rho(Su-\psi))$ is semi-smooth from $L^2(\Omega)^N$ to $L^2(\Omega)$. Further it is well known that it is Lipschitz continuous from $L^2(\Omega)^N$ to $L^2(\Omega)$. 
Since $S_\nu^*$ maps linear to $L^q(\Omega)$ we gain semi-smoothness and Lipschitz continuity of the whole second part.
\end{proof}

Let us analyze the problem in more detail. The Newton-derivative $D_N G$ of $G$ at the point $u_k$ is of the form
\[
D_N G(u_k) = \mathrm{Id} + \frac{1}{\alpha} \chi_{\Iset_k} D_N \tilde{F}(u_k).
\]
In order to compute the derivative of $\tilde F^\nu(u_k)$ recall
\[
\bar\mu(u_k) = (\mu+\rho(Su_k-\psi))_+.
\]
A Newton derivative of $\bar \mu(u_k)$ is given by
\begin{align*}
D_N \bmu(u_k) = \chi_{\Yset_k}\rho S \quad \text{ with } \quad  \Yset_k:=\left\lbrace x\in \Omega \colon (\mu+\rho(Su_k-\psi))(x) > 0 \right\rbrace.
\end{align*}
Hence, we obtain that the $\nu$-th component of the Newton derivative $D_N\tilde{F}(u_k)$ is given by
\[
D_N\tilde{F}(u_k)^\nu = S_\nu^\ast ( \chi_\nu S + \rho \chi_{\Yset_k} S ).
\]
 We end up with the following theorem.

\begin{theorem}
Let $G$ be given as in \eqref{eq:GforNewtonit}. A suitable Newton derivative of $G$ at $u_k$ in direction $h\in U$ is given by
\begin{align*}
(D_N G(u_k)h)^\nu = h_\nu+   \frac{1}{\alpha}  \chi_{\Iset_k^\nu} \left(S_\nu^*(\chi_\nu S+\chi_{\Yset_k}\rho S)h \right).
\end{align*}
\end{theorem}

Let us analyze this problem in more detail. In particular we want to provide a finite element discretization. Lets us recall the sets $\Inu_k$ and $\Yset_k$ and also define the sets $\Anua_k$ and $\Anub_k$
\begin{equation}\label{eq:defsets}
\begin{alignedat}{2}
\Anua_k&:=\left\lbrace x\in \Omega \colon -\frac{1}{\alpha}\tilde F^\nu(u_k)\leq u^\nu_a \right\rbrace, &\quad \Anub_k&:= \left\lbrace x\in \Omega \colon -\frac{1}{\alpha}\tilde F^\nu(u_k)\geq u^\nu_b\right\rbrace,\\
\Inu_k &:= \left\lbrace x\in \Omega \colon -\frac{1}{\alpha} \tilde F^\nu(u_k) \in(u^\nu_a,u^\nu_b) \right\rbrace, &\Yset_k&:=\left\lbrace x\in \Omega \colon (\mu+\rho(Su_k-\psi)) > 0 \right\rbrace.
\end{alignedat}
\end{equation}
Following the lines of Theorem \ref{theo:convSSNgen} we obtain that the following equality holds
\begin{equation*}
u^\nu_{k+1}(x)=\begin{cases}
u^\nu_a(x) \qquad &\text{if} \quad  x\in \Anua_k,\\
-\frac{1}{\alpha}\left(S_\nu^*(\chi_\nu (Su_{k+1}-y_d^\nu)+\chi_{\Yset_k}(\mu+\rho(Su_{k+1}-\psi)) \right)(x) &\text{if} \quad  x\in \Inu_k,\\
u^\nu_b(x) &\text{if} \quad  x\in \Anub_k.
\end{cases}
\end{equation*}

Thus, on the set $\Inu_k$ we obtain
\begin{equation}\label{eq:proj_u}
\chi_{\Inu_k}\left(u^\nu_{k+1}+\frac{1}{\alpha}\left(S_\nu^*(\chi_\nu (Su_{k+1}-y_d^\nu)+\chi_{\Yset_k}(\mu+\rho(Su_{k+1}-\psi))\right) \right) =0.
\end{equation}

Let us introduce the function $u_{k+1}^{\Iset}\in U$ with components $u^{\nu,\Iset}_{k+1}:= \chi_{\Iset_k^\nu}u^\nu_{k+1}$ for $\nu=1,...,N$. Hence, we can write $u_{k+1}^{\Iset} = \chi_{\Iset_k} u_{k+1}$ with $\chi_{\Iset_k} := \chi_{\Iset}(u_k)$ defined in \eqref{eq:defd}. In a similar way we define $\chi_{A_k^a}$ and $\chi_{A_k^b}$. Using this definitions we can now write \eqref{eq:proj_u} as a linear equation for the $\nu$-th component of $u_{k+1}^{\Iset}$ and we obtain
\begin{align*}
&u^{\nu,\Iset}_{k+1} + \frac{1}{\alpha}\chi_{\Inu_k}\left(S_\nu^*(\chi_\nu S+\chi_{\Yset_k}\rho S)u^{\Iset}_{k+1}\right) \notag\\
&\hspace{5em}= -\frac{1}{\alpha}\chi_{\Inu_k}\left( S_\nu^*((\chi_\nu S+\chi_{\Yset_k}\rho S)(\chi_{\Aset^{a}_k}u_a+\chi_{\Aset^{b}_k}u_b)-\chi_\nu y_d^\nu+\chi_{\Yset_k}(\mu-\rho \psi))\right).
\end{align*}

The Newton step can now be written in the following compact form.

\begin{lemma}
The solution $u_{k+1}$ of one step of the semi-smooth Newton method is given by
\[
u_{k+1} = u_{k+1}^{\Iset} + \chi_{A_k^a} u_a + \chi_{A_k^b} u_b,
\]
where $u_{k+1}^\Iset$ is given as the solution of the linear system
\begin{align}
\left(\mathrm{Id} + \chi_{\Iset_k} T^k\right) u^\Iset_{k+1} = \chi_{\Iset_k} g_k,
\label{eq:SSNNewtonStepsystem}
\end{align}
with the operator $T^k : U \to U$ and function $g_k \in U$ given by
\begin{align*}
(T^k h)^\nu &:= \frac{1}{\alpha} S_\nu^*(\chi_\nu S+\chi_{\Yset_k}\rho S)h,\\
(g_k)^\nu &: = -\frac{1}{\alpha}\chi_{\Inu_k}\left( S_\nu^*((\chi_\nu S+\chi_{\Yset_k}\rho S)(\chi_{\Aset^{a}_k}u_a+\chi_{\Aset^{b}_k}u_b)-\chi_\nu y_d^\nu+\chi_{\Yset_k}(\mu-\rho \psi))\right).
\end{align*}
Here $\mathrm{Id}:U \to U$ denotes the identity mapping.
\end{lemma}

The complete semi-smooth Newton method is given in the following algorithm.
\begin{algorithm}[H]
    \caption{Semi-smooth Newton method for problem \eqref{eq:augNEP_offset}}
    \label{alg:semi-smooth}
    \begin{algorithmic}
     \item[1:] Set $k=0$, choose $u_0$ in $\U^N$\\
     \item[2:] \textbf{repeat}
        \item[3:] Set $\Anua_k, \Anub_k, \Inu_k$ and $\Yset_k$ as defined in \eqref{eq:defsets} 
        \item[4:] Solve for $u_{k+1}^\Iset \in L^2(\Omega)^N$ by solving \eqref{eq:SSNNewtonStepsystem}
        \item[5:] Set $u_{k+1} := u_{k+1}^{\Iset} + \chi_{A_k^a} u_a + \chi_{A_k^b} u_b$
        \item[6:] Set $k:=k+1$
        \item[7:] \textbf{until} $\Anua_k=\Anua_{k-1},\Anub_k=\Anub_{k-1}, \Inu_k=\Inu_{k-1}$ and $\Yset_k=\Yset_{k-1}$.
    \end{algorithmic}
\end{algorithm}

\begin{theorem}[\bf{Convergence of the semi-smooth Newton method}] \label{theo:convSSN-PDE}
Let Assumption \ref{ass:alpha} hold and let $\bu$ denote the solution of \eqref{eq:augNEP_offset}. Then the semi-smooth Newton method from Algorithm \ref{alg:semi-smooth} has the following properties
\begin{itemize}
\item[a)] 
Let $\norm{u_0-\bu}_{L^2(\Omega)^N}$ be sufficiently small. Then the iterates $u_k$ converge for $k\rightarrow \infty$ superlinearly to $\bar{u}$ which is the solution of \eqref{eq:augNEP_offset}.
\item[b)] Let $u_k$ be generated by Algorithm \ref{alg:semi-smooth} such that the stopping criterion from step 7 is satisfied. Then $u_k$ is a solution of \eqref{eq:GforNewtonit}. 
\end{itemize}
\end{theorem}
\begin{proof}
\begin{enumerate}
\item[a)] 
Due to Theorem \ref{theo:convSSNgen} it remains to check if Assumption \ref{ass:bilinform} is satisfied. Thus, we consider the bilinear form 
  \begin{align*}
  a(w,v):=\left(\left(\mathrm{Id}+\frac{1}{\alpha}\chi_{\Iset_k}M \chi_{\Iset_k}\right)w,v\right)_U,
  \end{align*}
  where $(Mw)^\nu = S_\nu^\ast ( \chi_\nu S + \rho \chi_{\Yset_k} S )w$. Using the decomposition from the proof of Theorem \ref{thm:NEPunique} we obtain
\begin{align*}
&\left(\left(\mathrm{Id}+\frac{1}{\alpha}\chi_{\Iset_k}M\chi_{\Iset_k}\right)w,v\right)_U =
\sum_{\nu=1}^N\left(w^\nu +\frac{1}{\alpha} \chi_{\Iset_k}S_\nu^\ast ( \chi_\nu S + \rho \chi_{\Yset_k} S )\chi_{\Iset_k}w,w^\nu \right)_{L^2(\Omega)}\\
&=\norm{w}^2_U+\frac{1}{\alpha}\sum_{\nu=1}^N \left( S\chi_{\Iset_k}w,\chi_\nu S_\nu \chi_{\Iset_k}w^\nu\right)_{L^2(\Omega)} + \frac{\rho}{\alpha}\sum_{\nu=1}^N\left(\chi_{\Yset_k}S\chi_{\Iset_k}w, S_\nu\chi_{\Iset_k} w^\nu\right)_{L^2(\Omega)}\\
&=\norm{w}^2_U+\frac{1}{\alpha}\left( S\chi_{\Iset_k}w,\sum_{\nu=1}^N \chi_\nu S_\nu \chi_{\Iset_k}w^\nu\right)_{L^2(Z)} + \frac{\rho}{\alpha} \norm{\chi_{\Yset_k}S\chi_{\Iset_k}w}^2_U\\ 
&\geq\norm{w}^2_U+\frac{1}{\alpha}\norm{S\chi_{\Iset_k}w}^2_{L^2(Z)} -\frac{1}{\alpha}\left( S\chi_{\Iset_k}w,\sum_{\nu=1}^N (S_\nu-\chi_\nu S_\nu) \chi_{\Iset_k}w^\nu\right)_{L^2(Z)}\\
&\geq\norm{w}^2_U  -  \frac{1}{4\alpha}  \left\| \sum\limits_{\nu=1}^N \chi_Z \left( S_\nu - \chi_\nu S_\nu   \right)\chi_{\Iset_k}w^\nu\right\|_{L^2(\Omega)}^2 \\
&\geq \norm{w}^2_U  -\frac{1}{4\alpha}\left(\sum_{\nu=1}^N\norm{\chi_Z ( S_\nu - \chi_\nu S_\nu) }^2_{L^2(\Omega)\rightarrow L^2(\Omega)}\right)\left(\sum_{\nu=1}^N\norm{\chi_{\Iset_k}w^\nu}_{L^2(\Omega)}^2 \right) \\
&\geq \left(1 -\frac{1}{4\alpha}\sum_{\nu=1}^N\norm{\chi_Z ( S_\nu - \chi_\nu S_\nu) }^2_{L^2(\Omega)\rightarrow L^2(\Omega)} \right) \norm{w}^2_U.
\end{align*}
Choosing $\alpha$ as in Assumption \ref{ass:alpha} we get the desired result from Theorem \ref{theo:convSSNgen}.

\item[b)] We know that the solution of \eqref{eq:SSNNewtonStepsystem} is unique for fixed sets $\Anua_k, \Anub_k, \Inu_k$ and $\Yset_k$. \\We set $\Anua_k:= \Anua_{k+1}, \Anub_k:= \Anub_{k+1}, \Inu_k := \Inu_{k+1}$ and $\Yset_k:=\Yset_{k+1}$ in \eqref{eq:SSNNewtonStepsystem} and get
\begin{align*}
&u^{\nu,\Iset}_{k+1} + \frac{\chi_{\Inu_{k+1}}}{\alpha}\left(S_\nu^*(\chi_\nu S+\chi_{\Yset_{k+1}}\rho S)u^\Iset_{k+1}\right) =\\
& -\frac{\chi_{\Inu_{k+1}}}{\alpha}\left(\chi_\nu S_\nu^*(S+\chi_{\Yset_{k+1}}\rho S)(\chi_{\Aset^{a}_{k+1}}u_a+ \chi_{\Aset^{b}_{k+1}}u_b) -S_\nu^*(\chi_\nu y_d^\nu+\chi_{\Yset_{k+1}}(-\mu+\rho \psi))\right)\\
\Leftrightarrow & \ u^{\nu,\Iset}_{k+1} + \frac{\chi_{\Inu_{k+1}}}{\alpha}\left(S_\nu^*(\chi_\nu (Su_{k+1} -y_d^\nu) + \chi_{\Yset_{k+1}}\left(\mu + \rho(Su_{k+1} -\psi)\right)\right) = 0 \\
\Leftrightarrow &\ u^{\nu,\Iset}_{k+1} + \frac{\chi_{\Inu_{k+1}}}{\alpha} p^\nu(u_{k+1}) = 0 .
\end{align*}
Together with $u_{k+1}=u_a$ on $\Aset^a_k$ and $u_{k+1}=u_b$ on $\Aset^b_k$ we get
\begin{align*}
u^\nu_{k+1} - P_{[u^\nu_a,u^\nu_b]}\left(-\frac{1}{\alpha}p^\nu(u_{k+1}) \right) = 0.
\end{align*}
Hence, $u_{k+1}$ is a solution of \eqref{eq:GforNewtonit}.
\end{enumerate}
\end{proof}

Again, we can drop the assumption on $\alpha$ if the sets $\Omega_\nu$ coincide.

\begin{corollary}
Let $\Omega_\nu = \Omega_0 \subseteq \Omega$ for all $\nu=1,...,N$ and let $\alpha > 0$. Then the associated Newton method to the NEP \eqref{eq:augNEP_offset} converges superlinear.
\end{corollary}

Considering the non-reducible NEP \eqref{eq:prob_non_red} just results in setting $M:= S_\nu^*\chi_\nu S$ instead of  $S_\nu^*(\chi_\nu S+\rho\chi_{\Yset_k}S)$. Hence, the convergence result from Theorem \ref{theo:convSSN-PDE} transfers one by one to this kind of problem.

\subsection{Implementation}\label{sec:subsec:ImplNewIt}
Let us now focus on the details of an implementation using finite elements. 
To illustrate the implementation we focus on problem \eqref{eq:augNEP_offset} where $S$ denotes the solution operator of \eqref{eq:pdeDirichlet} with $A:=-\Delta$. Using standard methods the corresponding optimality system is given by
\begin{subequations}
\label{eq:JointConv-kkt-augNEP}
	\begin{equation}
	\begin{alignedat}{2}
	-\Delta \by &= \sum_{\nu=1}^N\bu^\nu &\quad &\text{in } \Omega\\
	\end{alignedat}
	\label{eq:JointConv-kkt-augNEP:1}
	\end{equation}
	\begin{equation}
	\begin{alignedat}{2}
	-\Delta \bp^\nu &= \chi_\nu(\by-y_d^\nu)+\bmu &\quad &\text{in } \Omega\\
	\end{alignedat}
	\label{eq:JointConv-kkt-augNEP:2}
	\end{equation}
	\begin{equation}
	( \bp^\nu+\alpha \bu^\nu,v^\nu-\bu^\nu) \geq 0\quad\forall v^\nu\in \Uadnu, \label{eq:JointConv-kkt-augNEP:3}
	\end{equation}
	\begin{equation}
	\bmu = \left(\mu+\rho(\by-\psi)\right)_+,\label{eq:JointConv-kkt-augNEP:4}
	\end{equation}
\end{subequations}
where the state and adjoint equation satisfy suitable boundary conditions. We are interested in a finite element discretization, so let us define the finite dimensional space $\mathbb{V}_h := \text{span}\{\phi_1,...,\phi_m\}$. The index $h$ indicates the underlying discretization and the functions $\phi_j$ denote the basis functions. 
 Let us now consider a discretized version of \eqref{eq:JointConv-kkt-augNEP}. We define the bilinear form
\[
a(w,v) := \int_{\Omega}\nabla w \nabla v \dx.
\] 

Then, the discretized version of \eqref{eq:JointConv-kkt-augNEP} is given by the solution $(y_h,u_h,p_h)$ of the system
\begin{equation}
\begin{alignedat}{3}
\label{eq:prob-discrete}
a(y_h,v_h) &= \left(\sum_{\nu=1}^N u_{h}^\nu,v_h \right) &\qquad &\forall v_h\in \mathbb{V}_h,\\
a(p_{h}^\nu,v_h) &= (\chi_\nu(y_h-y_d^\nu)+(\mu+\rho(y_h-\psi))_+,v_h)&& \forall v_h\in \mathbb{V}_h,\\
u_{h}^\nu &= P_{[u^\nu_a,u^\nu_b]}\left( -\frac{1}{\alpha} p_{h}^\nu\right).
\end{alignedat}
\end{equation}

Since for a given $u_h$ there exists an unique $y_h(u_h)$ and an unique adjoint states $p_h^\nu(u_h)$, system \eqref{eq:prob-discrete} can be reduced to the single equation
\begin{align*}
u_{h}^\nu &= P_{[u^\nu_a,u^\nu_b]}\left( -\frac{1}{\alpha} p_{h}^\nu (u_h)\right)\qquad \forall v_h\in \mathbb{V}_h.
\end{align*}
Again we define the active and inactive sets for the discrete function $u_{k,h}$:
\begin{alignat*}{2}
\Anua_k&:=\left\lbrace x\in \Omega \colon -\frac{1}{\alpha}p_h^\nu(u_{k,h})\leq u^\nu_a \right\rbrace,\qquad\quad   &\Anub_k&:= \left\lbrace x\in \Omega \colon -\frac{1}{\alpha}p_h^\nu(u_{k,h})\geq u^\nu_b\right\rbrace,\\
\Inu_k &:= \left\lbrace x\in \Omega \colon -\frac{1}{\alpha} p_h^\nu(u_{k,h}) \in(u^\nu_a,u^\nu_b) \right\rbrace, &\Yset_k&:=\left\lbrace x\in \Omega \colon (\mu+\rho(Su_{k,h}-\psi)) > 0 \right\rbrace.
\label{eq:defsets_discrete}
\end{alignat*}
We now define the functions $u_{k+1,h}^\Iset := \chi_{\Iset_k} u_{k+1,h}$, where $\chi_{\Iset^k}:=\chi_\Iset(u_k)$. Following the lines of the proof of Section \ref{sec:subsec:NewtonIt} we can establish a linear equation for the components of $u_{k+1,h}^\Iset$:
\begin{align*}
&u^{\nu,\Iset}_{k+1,h} + \frac{1}{\alpha}\chi_{\Inu_k}\left(S_\nu^*(\chi_\nu S+\chi_{\Yset_k}\rho S)u^{\Iset}_{k+1,h}\right) \notag\\
&\hspace{5em}= -\frac{1}{\alpha}\chi_{\Inu_k}\left( S_\nu^*((\chi_\nu S+\chi_{\Yset_k}\rho S)(\chi_{\Aset^{a}_k}u_a+\chi_{\Aset^{b}_k}u_b)-\chi_\nu y_d^\nu+\chi_{\Yset_k}(\mu-\rho \psi))\right).
\end{align*}

We want to solve this system by testing it with a function $v_h \in \mathbb{V}_h$. Note that we have $u_{k+1,h}^{\nu, \Iset} \not\in \mathbb{V}_h$ in general, but it can be calculated as a projection $u_{k+1,h}^{\nu, \Iset} = \chi_{\Iset_k^\nu} \tilde u_{k+1,h}^\nu$ of a function $\tilde u_{k+1,h}^\nu \in \mathbb{V}_h$, see \eqref{eq:prob-discrete}.

In the following denote $\underline{u_h} \in \R^m$ the coefficient vector of a function $u_h \in \mathbb{V}_h$, where $m$ denotes the dimension of the space $\mathbb{V}_h$.

Furthermore, we assume that $u_a^\nu, u_b^\nu \in \mathbb{V}_h$. We can reformulate the Newton step as a linear system in the coefficient vectors of $\tilde u_{k+1,h}^\nu$.

\begin{lemma}\label{lem:newton_discrete}
The coefficient vectors $\underline{\tilde u_{k+1,h}^{\nu}}$ for $1 \leq \nu \leq N$ satisfy the linear system
\begin{equation}\label{eq:mult_player_discrete}
\begin{pmatrix}
E_{1,1} & E_{1,2} & \dots & \dots & E_{1,N}\\
E_{2,1} & E_{2,2} & E_{2,3}   & \dots & E_{2,N}\\
\vdots & \ddots & \ddots &  &\vdots\\
\vdots &  &\ddots &  \ddots & E_{N-1,N}\\
E_{N,1} & \dots & \dots & E_{N,N-1} & E_{N,N}
\end{pmatrix}
\begin{pmatrix}
\underline{\tilde u_{k+1,h}^{1}}\\ \underline{\tilde u_{k+1,h}^{2}}\\ \vdots \\ \underline{\tilde u_{k+1,h}^{N}}
\end{pmatrix}
=
\begin{pmatrix}
C_1\\C_2\\ \vdots \\ C_{N-1}\\C_N
\end{pmatrix},
\end{equation}
where
\[
E_{i,j} := \left.\begin{cases}
M_{\Iset_k^i} + \frac{1}{\alpha}M_{\Iset_k^i}K^{-1}M_i K^{-1}M_{\Iset_k^j}+ \frac{\rho}{\alpha}M_{\Iset_k^i}K^{-1}M_{\Yset_k}K^{-1}M_{\Iset_k^j} & \text{if} \ i=j\\
\frac{1}{\alpha}M_{\Iset_k^i}K^{-1}M_i K^{-1}M_{\Iset_k^j}+ \frac{\rho}{\alpha}M_{\Iset_k^i}K^{-1}M_{\Yset_k}K^{-1}M_{\Iset_k^j} & \text{else}
\end{cases}\right\} \in \R^{m \times m},
\]
as well as
\begin{align*}
C_i &:= - \frac{1}{\alpha} M_{\Iset_k^i} K^{-1} \bigg[  M_{\Yset^k} ( \underline{\mu} - \rho \underline{\psi} ) - M_i \underline{y_d^i}\\
&\quad\quad + (M_i + \rho M_{\Yset^k}) K^{-1} \sum\limits_{i=1}^N \left(  M_{\Aset_k^{i,a}} \underline{u_a^i}  + M_{\Aset_k^{i,b}} \underline{u_b^i}\right)  \bigg] \in \R^m,
\end{align*}
and matrices $K, M_\nu, M_{\Inu_k}, M_{\Aset^{\nu,a}_k}, M_{\Aset^{\nu,b}_k}$ and $M_{\Yset_k}$ of the size $\R^{m\times m}$
 with
\[
\begin{array}{rlrl}
K_{ij}&:=\left[\int_{\Omega} \nabla \phi_i\cdot \nabla \phi_j\right]_{ij}, & (M_\nu)_{ij} &:= \left[\int_{\Omega_\nu}  \phi_i \phi_j\right]_{ij}\\
\left(M_{\Inu_k}\right)_{ij} &:= \left[\int_{\Inu_k}  \phi_i \phi_j\right]_{ij}, & \left(M_{\Aset^{\nu,a}_k}\right)_{ij}  &:= \left[\int_{\Aset^{\nu,a}_k}   \phi_i \phi_j\right]_{ij},\\
\left(M_{\Aset^{\nu,b}_k}\right)_{ij}  &:= \left[\int_{\Aset^{\nu,b}_k}   \phi_i \phi_j\right]_{ij}, & \left(M_{\Yset_k}\right)_{ij}  &:=\left[\int_{\Yset_k}   \phi_i \phi_j\right]_{ij},
\end{array}
\]

where $\phi_i,\phi_j$ denote the finite element basis functions of $\mathbb{V}_h$.
\end{lemma}

We can reconstruct the state and the adjoint states using the coefficient vectors $\underline{ \tilde u_{k+1,h}^{\nu}}$.

\begin{corollary}
The coefficient vector of the state $y_{k+1,h}$ satisfies
\[
\underline{ y_{k+1,h}  } = K^{-1}  \sum\limits_{\nu=1}^N   \left(  M_{\Iset_k^\nu} \underline{ \tilde u_{k+1,h}^{\nu}  } + M_{\Aset_k^{\nu,a}} \underline{u_a^\nu} + M_{\Aset_k^{\nu,b}} \underline{u_b^\nu}   \right)
\]
and the coefficient vector of the adjoint state $p_{k+1,h}^{\nu}$ can be computed by
\[
\underline{p_{k+1,h}^{\nu}} = K^{-1} \left(  M_\nu(  \underline{ y_{k+1,h}  }  - \underline{  y_d^\nu  }   ) + M_{\Yset_k}(  \underline{\mu} + \rho (  \underline{ y_{k+1,h}  }  - \underline{\psi} )   )   \right).
\]
\end{corollary}

The control $u_{k+1,h}^\nu$ can be computed by
\[
u_{k+1,h}^{\nu} = \chi_{\Iset_k^\nu} \tilde u_{k+1, h}^{\nu} + \chi_{\Aset_k^{\nu,a}} u_a^\nu + \chi_{\Aset_k^{\nu, b}} u_b^\nu.
\]
We only need the adjoint states to update our active sets, hence kinks and discontinuities in the control will not be accumulated during the algorithm. This is an advantage over the discrete version of the active-set method. However, the expressions arising in the Newton method are more complicated than the expressions in the active-set method.

\section{Active-Set Method}\label{sec:activeSet}
In this section we want to introduce an active-set method which is equivalent to the semi-smooth Newton method. For additional information regarding active-set methods, we want to refer to \cite{HintermItoKunisch2002primalDualActiveSetAsSSN, Stadler2009,ItoKunisch2003semismoothVarEq, ItoKunisch2003semismoothState,BergouItoKunusch1999primalDualConstOptCont} and the references therein.
\subsection{Equivalence to Active-Set Method}
 Let us establish the relation between the semi-smooth Newton method and the active-set method. For the sake of simplicity we consider the setting from Section \ref{sec:subsec:ImplNewIt}. We consider the problem's first-order optimality conditions \eqref{eq:JointConv-kkt-augNEP}.
Reformulating \eqref{eq:JointConv-kkt-augNEP:3} by applying the projection formula one has to solve systems of this type in the active-set method which is defined below.

\begin{algorithm}[H]
    \caption{Active-set method for problem \eqref{eq:augNEP_offset}}
    \label{alg:activeset}
    \begin{algorithmic}
     \item[1:] Set $k=0$, choose $(y_0,u_0,p_0)\in Y \times L^2(\Omega)^N \times L^2(\Omega)^N$\\
     \item[2:] \textbf{repeat}
        \item[3:] Set $\Anua_k, \Anub_k, \Inu_k$ and $\Yset_k$ as defined in \eqref{eq:defsets} 
        \item[4:] Solve for $(y_{k+1},u_{k+1},p_{k+1}) \in Y\times\U^N\times L^2(\Omega)^N$ by solving 	  
        \begin{subequations}
\begin{equation}
\begin{alignedat}{2}
-\Delta  y_{k+1}&=\sum_{\nu=1}^N u^\nu_{k+1} &\quad &\text{in } \Omega,\\
\end{alignedat}
\label{eq:alg:actset:1}
\end{equation}
\begin{equation}
\begin{alignedat}{2}
-\Delta p^\nu_{k+1} &= \chi_\nu\left(y_{k+1}-y_d^\nu\right)+\chi_{\Yset_k}(\mu+\rho(y_{k+1}-\psi)) &\quad &\text{in } \Omega,\\
\end{alignedat}
\label{eq:alg:actset:2}
\end{equation}
\begin{equation} 
u^\nu_{k+1}+\chi_{\Inu_k} \left(\frac{1}{\alpha} p^\nu_{k+1} \right)= \chi_{\Anua_kk}u^\nu_a + \chi_{\Anub_k}u^\nu_b
\label{eq:alg:actset:3}
\end{equation}
\end{subequations}
        \item [5:] Set $k=k+1$
        \item [6:] \textbf{until} $\Anub_k=\Anub_{k-1},\Anua_k=\Anua_{k-1}, \Inu_k=\Inu_{k-1}$ and $\Yset_k=\Yset_{k-1}$.
    \end{algorithmic}
\end{algorithm}
Here, we assume that the state and the adjoint equation satisfy suitable boundary conditions. Considering the active-set method from Algorithm \ref{alg:activeset}, equation \eqref{eq:alg:actset:2} yields the identity
\begin{align*}
p^\nu_{k+1} = S_\nu^*\left(\chi_\nu(y_{k+1}-y_d^\nu)+\chi_{\Yset_k}(\mu+\rho(y_{k+1}-\psi)\right).
\end{align*}
Inserting this identity in \eqref{eq:alg:actset:3} and exploiting $y_{k+1}=Su_{k+1}$ from \eqref{eq:alg:actset:1} we get
\begin{align*}
u^\nu_{k+1}&+ \frac{\chi_{\Inu_k}}{\alpha}\left(S_\nu^*(\chi_\nu S+\chi_{\Yset_k}\rho S)u_{k+1}\right) \notag\\
&= \chi_{\Anua_k}u^\nu_a+\chi_{\Anub_k}u^\nu_b+\frac{\chi_{\Inu_k}}{\alpha}\left(S_\nu^*(\chi_\nu y_d^\nu+\chi_{\Yset_k}(-\mu+\rho \psi))\right). 
\end{align*}
Using the representation $u_{k+1}= \chi_{\Aset^a_k} u_{k+1} +\chi_{\Aset^b_k} u_{k+1} +\chi_{\Iset_k}u_{k+1}$ we get $u_{k+1} =u_a$ on $\Aset^a_k$ and $u_{k+1} =u_b$ on $\Aset^b_k$. On the set $\Iset^k$ we have for all $\nu=1,..,N$
\begin{align*}
&\chi_{\Inu_k}u^\nu_{k+1}+ \frac{\chi_{\Inu_k}}{\alpha}\left(S_\nu^*(\chi_\nu S+\chi_{\Yset_k}\rho S)u_{k+1}\right)  =\frac{\chi_{\Inu_k}}{\alpha}\left(S_\nu^*(\chi_\nu y_d^\nu+\chi_{\Yset_k}(-\mu+\rho \psi))\right) \\
\Leftrightarrow & \chi_{\Inu_k}u^\nu_{k+1}+ \frac{\chi_{\Inu_k}}{\alpha}\left(S_\nu^*(\chi_\nu S+\chi_{\Yset_k}\rho S)\chi_{\Iset^k}u_{k+1}\right) \\
&\hspace{3em}=\frac{\chi_{\Inu_k}}{\alpha}\left(S_\nu^*(\chi_\nu y_d^\nu+\chi_{\Yset_k}(-\mu+\rho \psi))\right) -\frac{\chi_{\Iset^\nu_k}}{\alpha}\left(S_\nu^*(\chi_\nu S+\chi_{\Yset_k}\rho S)(\chi_{\Aset^a_k}u_a+\chi_{\Aset^b_k}u_b)\right),
\end{align*}

which coincides with a Newton step from \eqref{eq:SSNNewtonStepsystem}. In a similar way we can start with the Newton method and derive the active-set method. Hence, both methods are equivalent.

\subsection{Implementation}
Let us also present a numerical implementation of the active-set method. To formulate the active-set method we need to introduce the matrix $M\in\R^{m\times m}, (M)_{ij} := \left[  \int_{\Omega} \phi_i \phi_j  \right]_{ij}$.

\begin{lemma}
One step of the active-set method from Algorithm \ref{alg:activeset} can be computed by solving the system
\begin{equation}\label{eq:active_set_discrete}
\begin{pmatrix}
K & E_1 & 0\\
E_2 & 0 & E_3\\
0 & E_4 & E_5
\end{pmatrix}\left(\begin{array}{c}
\underline{y}\\
\underline{u}\\
\underline{p}
\end{array} \right)= \left(\begin{array}{c}
\mathbf{0}\\
l_1\\
l_2
\end{array}  \right)
\end{equation}
where $E_1 := \begin{pmatrix} -M& \cdots  & -M\end{pmatrix} \in \R^{m \times Nm}$ and
\[
\begin{array}{ll}
E_2 := \begin{pmatrix}-M_1-\rho M_{\Yset_k}\\ \vdots\\-M_N-\rho M_{\Yset_k} \end{pmatrix} \in \R^{Nm \times m}, \quad\quad&
E_3 := \begin{pmatrix}K&&\\&\ddots&\\&&K\end{pmatrix} \in \R^{Nm \times Nm} \vspace{0.25cm},\\
E_4 := \begin{pmatrix}M&&\\&\ddots&\\&&M\end{pmatrix} \in \R^{Nm \times Nm}, &
E_5 := \begin{pmatrix}\alpha^{-1} M_{\Iset_k^1}&&\\&\ddots&\\&& \alpha^{-1} M_{\Iset_k^N}\end{pmatrix} \in \R^{Nm \times Nm},
\end{array}
\]
as well as
\[
\underline{u} := \begin{pmatrix}\underline{u_{k+1,h}^1}\\\vdots\\\underline{u_{k+1,h}^N}\end{pmatrix} \in \R^{Nm}, \quad\quad  \underline{y} := \underline{y_{k+1,h}} \in \R^{m}, \quad\quad \underline{p} := \begin{pmatrix}\underline{p_{k+1,h}^1}\\\vdots\\\underline{p_{k+1,h}^N}\end{pmatrix} \in \R^{Nm},
\]
and right hand side
\[
l_1:= \begin{pmatrix}-M_1 \underline{y_d^1} + M_{\Yset_k} (\underline{\mu} - \rho \underline{\psi})\\\vdots\\-M_N \underline{y_d^N} + M_{\Yset_k} (\underline{\mu} - \rho \underline{\psi})\end{pmatrix}, \quad l_2 := \begin{pmatrix}M_{\Aset_k^{1,a}} \underline{u_a^1} + M_{\Aset_k^{1,b}} \underline{u_b^1}\\\vdots\\M_{\Aset_k^{N,a}} \underline{u_a^N} + M_{\Aset_k^{N,b}} \underline{u_b^N} \end{pmatrix}, \quad \mathbf{0} \in \R^m
\]
with the notation used in Lemma \ref{lem:newton_discrete}.
\end{lemma}

Let us now compare the discrete Newton step \eqref{eq:mult_player_discrete} and the discrete active-set method \eqref{eq:active_set_discrete}. The entries on the diagonal of the matrix on the left hand side of \eqref{eq:mult_player_discrete} $E_{\nu,\nu}$ are symmetric. However, for $N>1$ the resulting system is not symmetric.  Note that the matrix \eqref{eq:mult_player_discrete} should not be computed explicitly due to the appearance of $K^{-1}$. Still it is possible to compute  its matrix-vector multiplication. This makes it impossible to apply a direct solver or a preconditioner which is based on decomposition, i.e. LU-factorisation. However, it can be solved by iterative methods, i.e. GMRES or BiCGSTAB.
The resulting system for the active-set method \eqref{eq:active_set_discrete} is not symmetric even for $N=1$, but it can be solved by a direct solver with a preconditioner, i.e. incomplete LU-factorisation.

\section{Numerical Examples}
The matrices are computed using DOLFIN \cite{LoggWells2010a,LoggWellsEtAl2012a}, which is part of the open-source computing platform FEniCS \cite{AlnaesBlechta2015a,LoggMardalEtAl2012a}. The arising linear systems are solved with NumPy and SciPy.

\subsection{Example 1 - Four Player Game }
We consider a four player game like \eqref{eq:augNEP_offset} on the domain $\Omega = (0,1)^2$ with observation domains
\begin{alignat*}{6}
\Omega_1 &:= \left(0, \frac{1}{2} \right) \times  \left(0, \frac{1}{2} \right), &\quad & \Omega_2 := \left(\frac{1}{2},1 \right) \times  \left(0, \frac{1}{2} \right),\\
\Omega_3 &:= \left( \frac{1}{2},1 \right) \times  \left( \frac{1}{2},1 \right), & & \Omega_4 := \left(0, \frac{1}{2} \right) \times  \left( \frac{1}{2},1 \right).
\end{alignat*}
In this example we assume that $S$ is the solution mapping of the state equation $-\Delta y = \sum_{\nu=1}^N u^\nu$ with homogeneous Dirichlet boundary conditions.
The desired states are given by constant functions
\[
y_d^1 := 0, \;\; y_d^2 := 1, \; y_d^3 := 2, \; y_d^4 := 3 
\]
and we choose $\psi(x_1,x_2):= -2 x_1+2x_2+2$, where $(x_1,x_2)\in \Omega$. For the approximation of the multiplier we set $u_0,p_0$ and $\mu$ equal zero as well as $y_0 = 10.0$, $\alpha = 10^{-5}$, and $\rho= 10$.
Let us introduce the quantity
\[
\kappa(u_k) :=   \log\left( \frac{\|u_{k+1} - u_k\|_U}{\|u_k - u_{k-1}\|_U} \right) \left( \log\left( \frac{\|u_k - u_{k-1}\|_U}{\|u_{k-1} - u_{k-2}\|_U}\right) \right)^{-1},
\]
which is an approximation for the numerical order of convergence. If the sequence $(u_k)_k \subset U$ converges superlinear we expect $\kappa(u_k) \in (1,2)$ for $k$ large enough. Note that we do not have an exact solution available to compute the order of convergence, but in practice $\kappa(u_k)$ will give a good approximation. We use a regular triangulation with different mesh sizes $h$. 
We applied both, the semi-smooth Newton method and the active-set method to this type of problem. The system that arises if the active-set method is applied has been solved directly by using the \texttt{spsolve} method from the \texttt{scipy.sparse.linalg} library. The Newton equation instead has to be solved by an iterative method. Here we make use of the \texttt{gmres} method from the same library and use a tolerance of $10^{-12}$. Since both methods are equivalent it is not surprising that the approximated order of convergence $\kappa$ and the change of the active sets  coincide for both methods. Table \ref{tab:rates_ex1} shows the computed results dependent on $h$ for the active-set and the semi-smooth Newton method, respectively.  Clearly, the computed orders of convergence support the superlinear convergence.

We are using linear finite elements for the controls, adjoints and state variable. Let us quickly comment on our stopping criterion step 7 of Algorithm \ref{alg:semi-smooth} or step 6 of Algorithm \ref{alg:activeset}. Both algorithms stop when the active and inactive sets coincide. Due to the use of linear finite elements we compare the values on the nodes to check this condition. Let us illustrate this on the example of the set $\Aset_k^{\nu, a}$, which is defined by the inequality $\alpha^{-1} p_h^\nu(u_{k,h}) \leq u_a^\nu$. We now count all the nodes which lie in the symmetric difference of $\Aset_{k+1}^{\nu, a}$ and $\Aset_{k}^{\nu, a}$. If this returns zero, we conclude that $\Aset_{k+1}^{\nu, a} \approx \Aset_k^{\nu, a}$ holds good enough. We count these nodes for all the active and inactive sets in each iteration and sum them up. This calculation can be found in the row labeled "nodes".
\begin{figure}[H]
\centering
\includegraphics[scale=0.28]{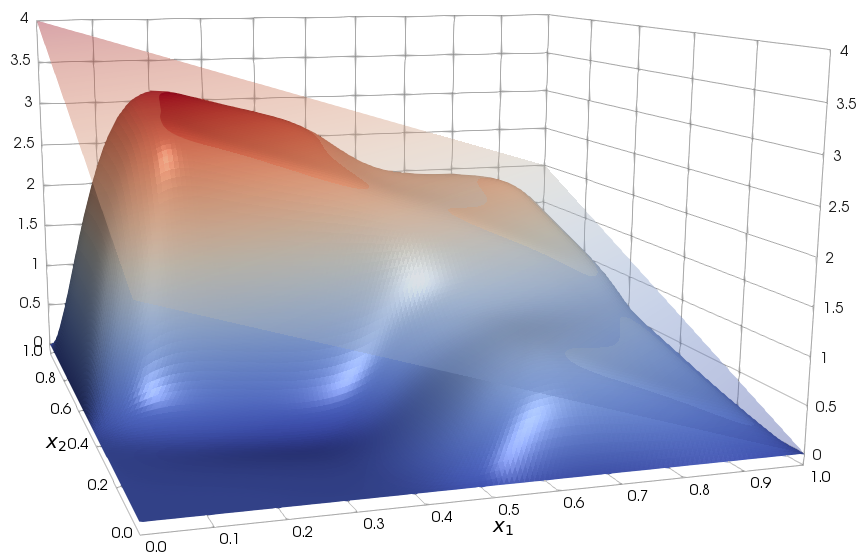}
\label{fig:Ex1-stateandconstraint}
\caption{(Example 1) Computed state and the state constraint (transparent).}
\end{figure}
\renewcommand{\arraystretch}{1.1}
\begin{table}[H]
\center
\small
\begin{tabular}{c|c|l|c|c|c|c|l|c|c|c}
& \multicolumn{5}{c|}{$h \approx 0.02,\ \text{dof}\approx 4\cdot 10^3$} & \multicolumn{5}{c}{$h\approx 0.01,\ \text{dof}\approx 1.6\cdot 10^4$}     \\\hline
 $k$ &  $\kappa(u_k)$ & nodes &  opt AS & opt N & gmres &  $\kappa(u_k)$ & nodes & opt AS & opt N & gmres\\\hline

1 &  		&2490 	& 1.8e-13	& 3.1e-07	&87& 		  &9543	& 1.8e-13	&4.9e-07   &84	\\
2 &  	 	&818	& 8.8e-14	& 1.7e-07	&74& 		  &3286	& 8.5e-14	&1.7e-07	&73	\\
3 &  	 	&417	& 5.2e-14 	& 1.3e-07	&66& 		  &1617	& 5.3e-14	&1.2e-07	&65	\\
4 & 1.3912 	&317	& 5.0e-14	& 1.0e-07	&61&  1.3883  &1256	& 5.0e-14	&1.5e-07	&59	\\
5 & 1.0706 	&179	& 5.1e-14	& 1.0e-07	&57&  1.0391  &701	& 5.2e-14	&1.5e-07	&55	\\
6 & 0.3861 	&102	& 5.3e-14	& 1.4e-07	&53&  0.4398  &380	& 5.3e-14	&1.3e-07	&52	\\
7 & 1.0543 	&48		& 5.1e-14	& 7.8e-08	&51&  0.7993  &138	& 5.5e-14	&1.3e-07	&49	\\
8 & 2.4179 	&13		& 5.2e-14	& 9.5e-08	&48&  2.7034  &78	& 5.4e-14	&1.0e-07	&47	\\
9 & 1.2989 	&6		& 5.1e-14	& 1.1e-07	&45&  1.3089  &22	& 5.4e-14	&1.0e-07	&44	\\
10 & 1.4880 &2		& 5.3e-14	& 1.1e-07	&42&  1.4354  &5	& 5.4e-14	&7.8e-08	&42	\\
11 & 1.7432 &0		& 5.2e-14	& 7.8e-08	&35&  1.6903  &1	& 5.5e-14	&7.5e-08	&35	\\
12 & 	    &		&  			&     		&  &  1.9467  &0	& 5.4e-14	&7.8e-08	&19	\\
\end{tabular}
\caption{(Example 1) Computed order of convergence $\kappa(u_k)$, change of nodes of the respective active sets, optimality of the problem, i.e., $\norm{u_k-P_\Uad(-\frac{1}{\alpha}p_k)}_{L^2(\Omega)}$ for the Newton method (opt N) and the active-set method (opt AS) and number of GMRES iterations for solving the Newton system.}
\label{tab:rates_ex1}
\end{table}

\subsection{Example 2 - Four Player Game with Known Exact Solution}
Next, we aim at solving \eqref{eq:augNEP_offset}, where $S$ denotes the solution operator of 
$$ -\Delta y = \sum_{\nu=1}^N u^\nu + f \quad \text{ in } \Omega, \qquad y = 0 \quad \text{ on } \partial \Omega,$$
where $f$ denotes a function in $L^2(\Omega)$. This setting differs slightly from the one presented above. However, it is easy to see that this does not have any impact on our convergence analysis. We investigate a four player game on the domain $\Omega = (-1,1)^2$ with observations domains
\begin{alignat*}{6}
\Omega_1 &:= \left(-1, 0 \right) \times  \left(-1, 0 \right), &\quad & \Omega_2 := \left(0,1 \right) \times  \left(-1, 0\right),\\
\Omega_3 &:= \left( -1,0 \right) \times  \left( 0,1 \right), & & \Omega_4 := \left(0, 1 \right) \times  \left( 0,1 \right).
\end{alignat*}

First, with $(x_1,x_2)\in \Omega$ we set the optimal state 
$$ \by(x_1,x_2) := \sin(2\pi x_1)\sin(2\pi x_2).$$
With $\xi^1 := (0.5, -0.5, 0.5, -0.5)$ and $\xi^2 := (0.5, 0.5, -0.5, -0.5)$ we set
$$	 r_\nu:= r_\nu(x_1,x_2) := \sqrt{(x_1+\xi^1_\nu)^2 + (x_2+\xi^2_\nu)^2}	$$
and define for $\nu= 1, ..., N$ the optimal adjoint states via
\begin{align*}
\bp_\nu := (-1)(- r_\nu^2 + 0.25)(16r_\nu^4 - 8r_\nu^2 + 1).
\end{align*}
Choosing a regularization parameter $\alpha$ and setting the control constraints $u_a^\nu:= -1.0$ and $u_b^\nu:= 20$ for all $\nu$ we construct the optimal control via $\bu^\nu := P_{\Uadnu}\left(-\frac{1}{\alpha}\bp^\nu\right)$. Due to the construction of the adjoint states we obtain $\bu^\nu = 0 $ in $\Omega\backslash \Omega_\nu$
We set $f:= -\Delta \by -\sum_{\nu=1}^N \bu^\nu$ so that $\by$ and $\bu^\nu$ satisfy the state equation. It remains to construct $y_d^\nu$, $\nu=1,...,N$. Due to the adjoint equation we obtain 
\begin{align*}
y_d^\nu := 
\begin{cases}
\by +\Delta\bp^\nu + (\mu+\rho(\by-\psi))_+   & \text{ in } \Omega_\nu, \\
0  & \text{ else.}
\end{cases}
\end{align*}
For our numerical experiments we use $\rho:= 10.0$, $\mu:= 0$ and $\psi:=2.0$. In order to solve this problem we apply the active-set method using the initial values $(y_0,u_0,p_0) := (10, 0, 0)$. Due to the knowledge of the exact solution the rate $R$ and order of convergence $\kappa$ can be estimated via 
\begin{align*}
\lim_{k\rightarrow\infty}\frac{\norm{u_{k+1}-\bu}_U}{\norm{u_k-\bu}_U} = R, \qquad \kappa^{ex}(u_k) = \left(  \log{\frac{\norm{u_{k+1}-\bu}_U}{\norm{u_k-\bu}_U}}  \right) \left(   \log{\frac{\norm{u_{k}-\bu}_U}{\norm{u_{k-1}-\bu}_U} }\right)^{-1}.
\end{align*} 
We solved the problems for $h\approx 0.02$ which corresponds to approximately $1.6\cdot 10^4$ degrees of freedom and used a tolerance of $10^{-8}$ for the \texttt{gmres} method. For determining the rate of convergence we compute in each iteration $R(u_k):=\frac{\norm{u_{k+1}-\bu}_U}{\norm{u_k-\bu}_U}$ and denote the corresponding value of the active-set method by $R_{AS}(u_k)$ and the one of the semi-smooth Newton method by $R_N(u_k)$. Let us check on the convergence properties corresponding to different regularization parameters $\alpha$. For $\alpha<0.002$ neither the active-set method nor the semi-smooth Newton method converged. For $\alpha=0.1$ the upper constraint $u_b^\nu$ is not active. Hence, we choose $\alpha\in[0.002,0.1]$. As can be seen in Table \ref{tab:rates_ex2} the computed values $R(u_k)$ for both methods imply superlinear convergence until only very few nodes in the active set change. The result is strengthened by the corresponding computed orders of convergence, i.e., $\kappa^{ex}_{AS}(u_k)$ for the active-set method and $\kappa^{ex}_N(u_k)$ for the semi-smooth Newton method. In contrast to Example 1 the results of the semi-smooth Newton method now differ slightly from the active-set method. 
This may be due to the arising expression $Sf$ in the Newton equation, which requires an additional solution of the corresponding PDE. Indeed, by introducing $\tilde{y}_d^\nu:=y_d^\nu - Sf$ we obtain that the tracking type term 
\[\frac{1}{2}\norm{y-y_d^\nu}^2_{L^2(\Omega_\nu)} = \frac{1}{2}\norm{Su + Sf -y_d^\nu}^2_{L^2(\Omega_\nu)} = \frac{1}{2}\norm{Su - \tilde{y}_d^\nu}^2_{L^2(\Omega_\nu)}\]
can be treated in the same way as in Section \ref{sec:NewtonNonRegNEP}. The same approach has to be followed for the state constraint $\psi$, where we set $\tilde{\psi}:= \psi- Sf$. Since $\tilde{y}_d^\nu$ and $\tilde{\psi}$ appear in every iteration on the right hand side of the Newton equation, the error, which arises due to the solution of $Sf$ will be accumulated over the iterations. 
Finally, Figure \ref{fig:ex2} depicts the sum of the computed controls and the computed state.
\newpage
\renewcommand{\arraystretch}{1.1}
\begin{table}[H]
\small
\center
\begin{tabular}{c|c|c|l|c|c|l|c}
 & \multicolumn{7}{c}{$\alpha = 0.002$} \\\hline
 $k$ & $R_{AS}(u_k)$ & $\kappa^{ex}_{AS}(u_k)$ & nodes AS & $R_N(u_k)$ & $\kappa^{ex}_N(u_k)$ & nodes N & gmres  \\\hline
1 &   		&  		 	& 94091	&		& 		&	94091	& 	23\\
2 & 0.6210	&  		 	& 75566	&0.6211	& 		&	75570	& 	104\\
3 & 0.5949	&  	1.0902	& 25816	&0.5948	&1.0907 &	25106	& 	149\\
4 & 0.8378	&  	0.3408	& 49534	&0.8397 &0.3363  & 	49726 	& 	155\\
5 & 0.2105	& 	8.8053	& 22188 &0.2124	&8.8686 &  19272	& 	113\\
6 & 0.1897	&  	1.0667	& 8368  &0.2349	&0.9353 & 	11216	& 	105\\
7 & 0.0585	&  	1.7082	& 32  	&0.0449 &2.1426 & 	56		& 	86\\
8 & 1.0005	&  -0.0002	& 0  	&0.9736	&0.0086 & 	0		& 	67\\
\multicolumn{8}{c}{} \\

 & \multicolumn{7}{c}{$\alpha = 0.005$} \\\hline
 $k$ & $R_{AS}(u_k)$ & $\kappa^{ex}_{AS}(u_k)$ & nodes AS & $R_N(u_k)$ &$\kappa^{ex}_N(u_k)$ & nodes N & gmres  \\\hline
1 &   		&  		 	& 	41917	&	 		&		 & 41909	& 18 	\\
2 &  0.6574	&  		 	&	28364	&	0.6574 	&		 & 28340	& 81	\\
3 &  0.5222	&  1.5488	&	32196	&	0.5221	& 1.5492 & 32132	& 106	\\
4 & 0.4825 	&  1.1218	& 	14404	&	0.4824 	&1.1218	 & 14492	& 102\\
5 & 0.3000	&  1.6519	&  	19696	&	0.2998 	&1.6522	 & 19656 	& 99\\
6 & 0.0190	&  3.2912	&   120		&	0.0184 	&3.3152	 & 120		& 79\\
7 & 0.7412	&  0.0756	&   0		&	0.7346 	&0.0772	 & 	8		& 61	\\
8 & 		&		    & 			& 	1.000 	&2.7e-06& 0 		& 3 \\
\multicolumn{8}{c}{} \\

 & \multicolumn{7}{c}{$\alpha = 0.01$} \\\hline
 $k$ & $R_{AS}(u_k)$ & $\kappa^{ex}_{AS}(u_k)$ & nodes AS & $R_N(u_k)$ &$\kappa^{ex}_N(u_k)$ & nodes N & gmres  \\\hline
1 &   		&  		 	& 29083 &	 	 &  		& 16	& 29083	\\
2 &  0.5903 &  		 	& 13496	&0.5904	 &  		& 62	& 13512	\\
3 &  0.5924 &  0.9933 	& 21940	&0.5936	 & 0.9896	& 67	& 21924	\\
4 & 0.6014 	&  0.9712	& 9486	&0.6005	 & 0.9780	& 68	& 9502	\\
5 & 0.1720	&  3.4612	& 8200  &0.1722	 & 3.4495 	& 77 	& 8208 	\\
6 & 0.0204	&  2.2117	& 96  	&0.02029 & 2.2155 	& 52	& 96\\
7 & 1.0030 	&  -0.0008  & 0  	&1.0033	 & -0.0008 	& 0		& 28\\
\multicolumn{8}{c}{} \\

 & \multicolumn{7}{c}{$\alpha = 0.1$} \\\hline
 $k$ & $R_{AS}(u_k)$ & $\kappa^{ex}_{AS}(u_k)$ & nodes AS & $R_N(u_k)$ &$\kappa^{ex}_N(u_k)$ & nodes N & gmres  \\\hline
1 &   		&  		 	& 12182 &	 	 &			& 12182	 &10\\
2 & 0.4016	&  		 	& 3413	&0.4016	 &			& 3413	 &10	\\
3 & 0.6011	&  0.5579	& 1046	&0.6011	 & 0.5579	& 1046	 &10 	\\
4 & 0.0014 	&  12.9780	& 0		&0.0014	 & 12.9780	& 0		 &11\\
\end{tabular}
\caption{(Example 2) Computed rates $R_{AS}(u_k)$, $R_N(u_k)$ and order of convergence $\kappa^{ex}_{AS}(u_k)$,$\kappa^{ex}_{N}(u_k)$, change of nodes of the respective active sets and number of gmres iterations for solving the Newton system.}
\label{tab:rates_ex2}
\end{table}
\begin{figure}[H]
\includegraphics[width=0.48\textwidth]{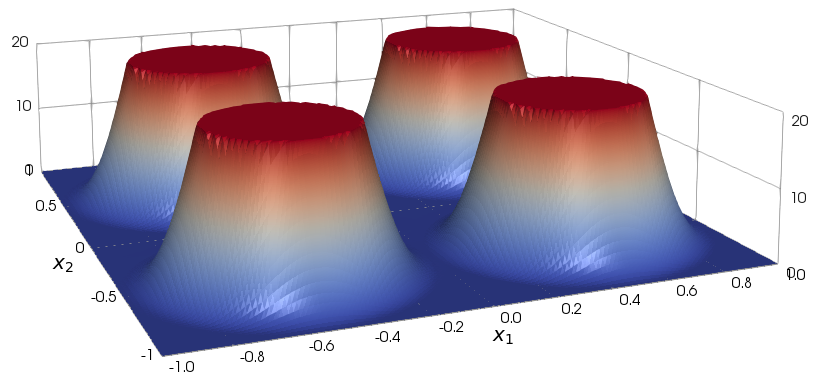}
\hfill
\includegraphics[width=0.48\textwidth]{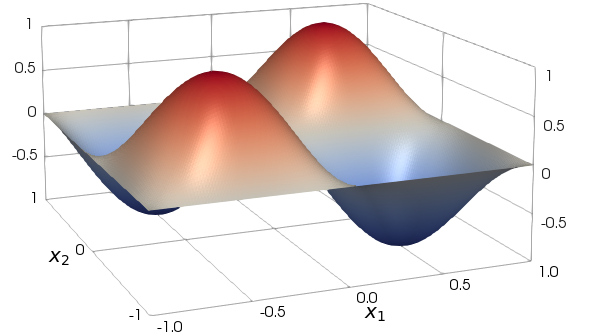}
\caption{(Example 2) Left: Computed sum of controls, right: computed state.}
\label{fig:ex2}
\end{figure}

\newpage
\titleformat{\section}[runin]{\bfseries}{}{1em}{}[. ]
\titleformat{\subsection}[runin]{\bfseries}{}{1em}{}[. ]
\section*{Appendix}
We present the proof of Theorem \ref{theo:convSSNgen}.
\begin{proof}
Due to Assumption \ref{ass:Fpart1} the operator $G$ satisfies the approximation condition. We are left to check the regularity condition. By Lemma \ref{lem:chainrule} we know that a Newton derivative of $G$ at the point $u_k$ is given by
\begin{align*}
D_N G(u_k) = \textrm{Id}+ \frac{1}{\alpha} \chi_{\Inu_k} M
\end{align*}
where $\chi_{\Inu_k}$ is defined as in \eqref{eq:defd} and $M := D_N \tilde F(u_k)$. Note that $\chi_{\Inu_k}$ depends on $u_k$. We want to apply the Lax-Milgram theorem to obtain boundedness of $D_N G(u_k) ^{-1}$. However, a direct application to the bilinear form $( D_N G(u_k)w,v )_U$ is not possible since it is not coercive in general. We consider instead an equivalent formulation which satisfies the assumptions of the Lax-Milgram theorem.

We denote by $u_{k+1}:=u_k+\delta_k$ the next iterate of the semi-smooth Newton method. A Newton step is given by
\begin{align*}
& D_N G(u_k)\delta_k = - G(u_k)\\
\Leftrightarrow\qquad  &\left( \mathrm{Id}+ \frac{1}{\alpha} \chi_{\Inu_k} M\right) \delta_k = - u_k + P_{\Uad}\left(-\frac{1}{\alpha}\tilde{F}(u_k)\right)\\
\Leftrightarrow\qquad  & \left( \mathrm{Id}+ \frac{1}{\alpha}\chi_{\Inu_k} M \right)  u_{k+1} =  P_{\Uad}\left(-\frac{1}{\alpha}\tilde{F}(u_k)\right) +\frac{1}{\alpha} \chi_{\Inu_k} M u_k.
\end{align*}

Using this representation we see that
\begin{equation*}\label{eq:reconstruct_control}
u^\nu_{k+1}(x)=\begin{cases}
u_a^\nu(x) \qquad &\text{if} \quad  x\in \Anua_k,\\
\left(-\frac{1}{\alpha}\tilde{F}(u_k)-\frac{1}{\gamma}M\delta_k\right)(x) &\text{if} \quad  x\in \Inu_k,\\
u_b^\nu(x) &\text{if} \quad  x\in \Anub_k.
\end{cases}
\end{equation*}
where the sets $\Anua_k,\Anub_k,\Inu_k$ are defined by
\begin{equation*}
\begin{split}
\Anua_k&:=\left\lbrace x\in \Omega \colon -\frac{1}{\alpha}\tilde{F}(u_k)^\nu(x)\leq u_a^\nu(x) \right\rbrace,\\
\Inu_k &:= \left\lbrace x\in \Omega \colon -\frac{1}{\alpha}\tilde{F}(u_k)^\nu(x) \in(u_a^\nu(x),u_b^\nu(x)) \right\rbrace,\\
\Anub_k&:= \left\lbrace x\in \Omega \colon -\frac{1}{\alpha}\tilde{F}(u_k)^\nu(x)\geq u_b^\nu(x)\right\rbrace.
\end{split}
\end{equation*}

Similar to $\chi_{\Inu_k}$ we define the sets $\chi_{\Aset^a_k}$ and $\chi_{\Aset^b_k}$ in a componentwise manner based on the set $\Anua_k$ and $\Anub_k$. Using the decomposition $u_{k+1} = \chi_{\Iset_k}u_{k+1} +(1-\chi_{\Iset_k}) u_{k+1}$
and exploiting the identities
\begin{align*}
P_{\Uad}\left(-\frac{1}{\alpha}\tilde{F}(u_k)\right) &= \chi_{\Aset^a_k}  u_a + \chi_{\Aset^b_k}  u_b +\chi_{\Iset_k} \left(-\frac{1}{\alpha}\tilde{F}(u_k)\right),\\
(1-\chi_{\Iset_k})u_{k+1} &= \chi_{\Aset^a_k}  u_a +\chi_{\Aset^b_k}  u_b,
\end{align*}
we have the equivalent formulation
\begin{align*}
\bigg( \mathrm{Id}&+ \frac{1}{\alpha} \chi_{\Iset_k} M \bigg)  \chi_{\Iset_k} u_{k+1}\\
&=  P_{\Uad}\left(-\frac{1}{\alpha}\tilde{F}(u_k)\right) +\frac{1}{\alpha} \chi_{\Iset_k} Mu_k -\left( \mathrm{Id}+ \frac{1}{\alpha} \chi_{\Iset_k} M\right) \bigg (1-\chi_{\Iset_k} \bigg) u_{k+1}\\
&= \chi_{\Iset_k}\left(-\frac{1}{\alpha}\tilde{F}(u_k)\right) + \frac{1}{\alpha} \chi_{\Iset_k} Mu_k  -\frac{1}{\alpha}\chi_{\Iset_k} M \bigg (1-\chi_{\Iset_k} \bigg) u_{k+1}\\
&= \chi_{\Iset_k}\left(-\frac{1}{\alpha}\tilde{F}(u_k)\right) + \frac{1}{\alpha}\chi_{\Iset_k} Mu_k  -\frac{1}{\alpha} \chi_{\Iset_k} M \bigg( \chi_{\Aset^a_k}  u_a +\chi_{\Aset^b_k}  u_b\bigg).
\end{align*}

Applying the decomposition $u_{k+1} = \chi_{\Iset_k} u_{k+1} + \chi_{\Aset^a_k}  u_a +\chi_{\Aset^b_k}  u_b$ we finally reach at

\begin{align*}
\left( \mathrm{Id}+ \frac{1}{\alpha} \chi_{\Iset_k} M \chi_{\Iset_k}\right) u_{k+1} = &\chi_{\Iset_k}\left(-\frac{1}{\alpha}\tilde{F}(u_k)\right) + \frac{1}{\alpha} \chi_{\Iset_k} Mu_k\\
&  -\frac{1}{\alpha} \chi_{\Iset_k} M \bigg( \chi_{\Aset^a_k}  u_a +\chi_{\Aset^b_k}  u_b\bigg)\\
&+ \chi_{\Aset^a_k}  u_a +\chi_{\Aset^b_k}  u_b.
\end{align*}

Defining the bilinear form $a(w,v) = \left(\left( \mathrm{Id}+ \frac{1}{\alpha} \chi_{\Iset_k} M \chi_{\Iset_k}\right) w ,v\right)_U$ we get
\begin{align*}
a(w,v) & = (w,v)_U + \frac{1}{\alpha}  (\chi_{\Iset_k}M \chi_{\Iset_k} w,v)_U \leq c\norm{w}_U\norm{v}_U.
\end{align*}

Further, with Assumption \ref{ass:bilinform} we see that $a(w,v)$ satisfies the conditions of the Lax-Milgram Theorem, which yields boundedness of $\norm{D_N G(u_k) ^{-1}}_{U\rightarrow U}$.
\end{proof}

\bibliography{mybib}
\bibliographystyle{plain_abbrv}

\end{document}